\documentclass{article}[10]
\usepackage{amsmath,amssymb,amsthm}
\title{Stable determination of an immersed body \\ in a stationary Stokes fluid}
\author{ Andrea Ballerini \footnote{SISSA-ISAS, Via Beirut 2-4, 34151 Trieste, Italy. E-mail: balleand@sissa.it.}}
\date{}
\newcommand{\tmop}[1]{\ensuremath{\operatorname{#1}}}
\newcommand{\tmtextbf}[1]{{\bfseries{#1}}}
\newcommand{\tmtextit}[1]{{\itshape{#1}}}

\newtheorem{theorem}{Theorem}[section]
\newtheorem{lemma}[theorem]{Lemma}
\newtheorem{proposition}[theorem]{Proposition}

\newenvironment{definition}[1][Definition]{\begin{trivlist}
\item[\hskip \labelsep {\bfseries #1}]}{\end{trivlist}}
\newenvironment{remark}[1][Remark]{\begin{trivlist}
\item[\hskip \labelsep {\bfseries #1}]}{\end{trivlist}}

\newcommand{\dive}{\mathrm{div} {\hspace{0.25em}}}
\newcommand{\elledue}[1]{{\bf L}^2({#1})}
\newcommand{\accauno}[1]{{\bf H}^1({#1})}
\newcommand{\accan}[2]{{\bf H}^{#1}({#2})}
\newcommand{\ide}{{\hspace{0.25em}}\mathbb{I}}
\newcommand{\til}[1]{\widetilde{#1}}
\newcommand{\omegad}{{\Omega} {\setminus} \overline{D}}
\newcommand{\norma}[3]{\|#1\|_{\accan{#2}{#3}}}
\newcommand{\normadue}[2]{\|#1\|_{\elledue{#2}}}
\newcommand{\accano}[2]{{\bf H}^{#1}_0({#2})}
\newcommand{\nor}[2]{ \|#1 \|_{#2} }
\numberwithin{equation}{section}
\begin{document}

\maketitle
\begin{abstract} 
We consider the inverse problem of the detection of a single body, immersed in a bounded container filled with a fluid which obeys the Stokes equations, from a single measurement of force and velocity on a portion of the boundary. We obtain an estimate of stability of log-log type.
\end{abstract}
{\bf Mathematical Subject Classification (2010):} Primary 35R30. Secondary 35Q35, 76D07, 74F10. \\
{\bf Keywords:} Cauchy problem, inverse problems, Stokes system, stability estimates.

\section{Introduction.}
In this paper we deal with an inverse problem associated to the Stokes system. We consider $\Omega \subset \mathbb{R}^n$, with $n=2,3$, with a sufficiently smooth boundary $\partial \Omega$. We want to detect an object $D$ immersed in this container, by collecting measurements of the velocity of the fluid motion and of the boundary forces, but we only have access to a portion $\Gamma$ of the boundary $\partial \Omega$.
The fluid obeys the Stokes system in $\omegad$:
\begin{equation}
  \label{NSE} \left\{ \begin{array}{rl}
    \dive\sigma(u,p) &= 0 \hspace{2em} \mathrm{\tmop{in}} \hspace{1em}
    \omegad,\\
    \dive u & = 0 \hspace{2em} \mathrm{\tmop{in}} \hspace{1em} \omegad,\\
    u & = g \hspace{2em} \mathrm{\tmop{on}} \hspace{1em} \Gamma,\\
    u & = 0 \hspace{2em} \mathrm{\tmop{on}} \hspace{1em} \partial D.
  \end{array} \right.
\end{equation}
Here, \begin{displaymath} \sigma (u, p) =  \mu ( \nabla u + \nabla u ^T )  -  p \ide   \end{displaymath} is the \tmtextit{stress tensor}, where $\ide$ denotes the $n \times n$ identity matrix, and $\mu$ is the viscosity function. The last request in (\ref{NSE}) is the so called ``no-slip condition''. We will always assume constant viscosity, $\mu(x)=1$, for all $x \in \omegad$. We observe that if $(u,p) \in \accauno{\omegad} \times L^2(\omegad)$ solves (\ref{NSE}), then it also satisfies \begin{displaymath}
\triangle u -\nabla p=0.
\end{displaymath} 
Call $\nu$ the outer normal vector field to $\partial \Omega$.
The ideal experiment we perform is to assign $g \in \accan{\frac{3}{2}}{\Gamma}$ and measure on $\Gamma$ the normal component of the stress tensor it induces, \begin{equation}\label{psi}\sigma (u, p) \cdot \nu = \psi, \end{equation} 
and try to recover $D$ from a single pair of Cauchy data $(g, \psi)$ known on the accessible part of the boundary $\Gamma$. Under the hypothesis of $\partial \Omega$ being of Lipschitz class, the uniqueness for this inverse problem has been shown to hold (see \cite{ConcOrtega2}) by means of unique continuation techniques.  For a different inverse problem regarding uniqueness of the viscosity function $\mu$, an analogous uniqueness result has been shown to hold, under some regularity assumptions (see \cite{HXW}). \\
The stability issue, however, remains largely an open question. There are some partial "directional stability" type result, given in \cite{ConcOrtega} and \cite{ConcOrtega2}. This type of result, however, would not guarantee an a priori uniform stability estimate for the distance between two domains that yield boundary measurement that are close to each other. In the general case, even if we add some a priori information on the regularity of the unknown domain, we can only obtain a weak rate of stability. This does not come unexpected since, even for a much simpler system of the same kind, the dependence of $D$ from the Cauchy data is at most of logarithmic type. See, for example, \cite{ABRV} for a similar problem on electric conductivity, or \cite{MRC}, \cite{MR} for an inverse problem regarding elasticity. 
The purpose of this paper is thus to prove a log-log type stability for the Hausdorff distance between the boundaries of the inclusions, assuming they have $C^{2,\alpha}$ regularity.  Such estimates have been estabilished for various kinds of elliptic equations, for example, \cite{ABRV}, \cite{AlRon}, for the electric conductivity equation, \cite{MRC} and \cite{MR} for the elasticity system and the detection of cavities or rigid inclusions. For the latter case, the optimal rate of convergence is known to be of log type, as several counterexamples (see \cite{Aless1} and \cite{DiCriRo})  show.
The main tool used to prove stability here and in the aforementioned papers (\cite{ABRV}, \cite{MRC}, \cite{MR}) is essentially a quantitative estimate of continuation from boundary data, in the interior and in the boundary, in the form of a three spheres inequality, see Theorem \ref{teotresfere}, and its main consequences. However, while in \cite{ABRV} the estimates are of log type for a scalar equation, here, and in \cite{MRC} and \cite{MR}, only an estimate of log-log type could be obtained for a system of equations. The reason for this is that, at the present time, no doubling inequalities at the boundary for systems are available, while on the other hand they are known to hold in the scalar case. \\
The basic steps of the present paper closely follows \cite{MRC}, \cite{MR}, and are the following: \begin{enumerate} \item {\it An estimate of propagation of smallness from the interior}. The proof of this estimate relies essentially on the three spheres inequality for solutions of the bilaplacian system. Since both the Lam\'e system and the Stokes system can be represented as solutions of such equations (at least locally and in the weak sense, see \cite{GAES} for a derivation of this for the elasticity system), we expected the same type of result to hold for both cases. 
  \item {\it A stability estimate of continuation from the Cauchy data}. This result also relies heavily on the three spheres inequality, but in order to obtain a useful estimate of continuation near the boundary, we need to extend a given solution of the Stokes equation a little outside the domain, so that the extended solution solves a similar system of equation. Once the solution has been properly extended, we may apply the stability estimates from the interior to the extended solution and treat them like estimates near the boundary for the original solution. \item{\it An extension lemma for solutions to the Stokes equations}. This step requires finding appropriate conditions on the velocity field $u$ as well as for the pressure $p$ at the same time, in order for the boundary conditions to make sense. In Section 5 we build such an extension. We point out that, if we were to study the inverse problem in which we assign the normal component $\psi$ of the stress tensor and measure the velocity $g$ induced on the accessible part of the boundary, the construction we mentioned would fail to work.

\end{enumerate} 
The paper is structured as follows. In Section 2, we state the apriori hypotheses we will need throughout the paper, and state the main result, Theorem \ref{principale}. In Section 3 we state the estimates of continuation from the interior we need, Propositions \ref{teoPOS}, \ref{teoPOSC}, and Propositions \ref{teostabest} and \ref{teostabestimpr} which deal, in turn, with the  stability estimates of continuation from Cauchy data and a better version of the latter under some additional regularity hypotheses, and we use them for the proof of Theorem \ref{principale}.
In section 4, we prove Proposition \ref{teoPOS} and \ref{teoPOSC} using the three spheres inequality, Theorem \ref{teotresfere}. Section 5 is devoted to the proof of Proposition \ref{teostabest}, which will use an extension argument, Proposition \ref{teoextensionNSE}, which will in turn be proven in Section 6.

\section{The stability result.}
\subsection{Notations and definitions.}

Let $x\in \mathbb{R}^n$. We will denote by $ B_{\rho}(x)$ the ball in $\mathbb{R}^n$ centered in $x$ of radius $\rho$. We will indicate  $x = (x_1, \dots ,x_n) $ as $x= (x^\prime, x_n)$ where $x^\prime = (x_1 \dots x_{n-1})$. Accordingly, $B^\prime_{ \rho}(x^\prime)$ will denote the ball of center $x^\prime$ and radius $\rho$ in $\mathbb{R}^{n-1}$. 
We will often make use of the following definition of regularity of a domain.
\begin{definition} 
Let $\Omega \subset \mathbb{R}^n$ a bounded domain. We say $\Gamma \subset \partial \Omega$  is of class $C^{k, \alpha}$  with constants $\rho_0$, $M_0 >0$, where $k$ is a nonnegative integer, $\alpha \in [ 0,1 )$ if, for any $P \in \Gamma$ there exists a rigid transformation of coordinates in which $P = 0$ and  
\begin{equation} \label{regolarita}
\Omega \cap B_{\rho_0}(0) = \{ (x^\prime, x_n) \in  B_{\rho_0}(0) \, \, \mathrm{s.t. } \, \, x_n > \varphi (x^\prime)\},
\end{equation}
where $\varphi$ is a real valued function of class $C^{k, \alpha}(B^\prime_{\rho_0}(0))$ such that \begin{displaymath} \begin{split}  \varphi(0)&=0, \\ \nabla\varphi(0)&=0, \text{  if  } k \ge 1 \\ \| \varphi\|_{C^{k, \alpha}(B^\prime_{\rho_0}(0))} &\le M_0 \rho_0. 
\end{split}
\end{displaymath}
\end{definition}
When $k=0$, $\alpha=1$ we will say that $\Gamma$ is {\it of Lipschitz class with constants $\rho_0$, $M_0$}.

\begin{remark} We normalize all norms in such a way they are all dimensionally equivalent to their argument and coincide with the usual norms when $\rho_0=1$. In this setup, the norm taken in the previous definition is intended as follows:
\begin{displaymath}
\| \varphi\|_{C^{k, \alpha}(B^\prime_{\rho_0}(0))} = \sum_{i=0}^{k} \rho_0^i \| D^i \varphi\|_{L^{\infty}(B^\prime_{\rho_0}(0))} + \rho_0^{k+\alpha}  | D^k \varphi |_{\alpha,B^\prime_{\rho_0}(0) },
\end{displaymath}
where $| \cdot |$ represents the $\alpha$-H\"older seminorm 

\begin{displaymath}
| D^k \varphi |_{\alpha,B^\prime_{\rho_0}(0) } = \sup_{x^\prime, y^\prime \in B^\prime_{\rho_0}(0), x^\prime \neq y^\prime }  \frac{| D^k \varphi(x^\prime)-D^k \varphi(y^\prime)| }{|x^\prime -y^\prime|^\alpha},
\end{displaymath}
and $D^k \varphi=\{ D^\beta\varphi\}_{|\beta|= k}$ is the set of derivatives of order $k$.
Similarly we set 
\begin{displaymath}
\normadue{u}{\Omega}^2 = \frac{1}{\rho_0^n}   \int_\Omega u^2 \,
\end{displaymath}

\begin{displaymath}
\norma{u}{1}{\Omega}^2 = \frac{1}{\rho_0^n} \Big( \int_\Omega u^2 +\rho_0^2 \int_\Omega |\nabla u|^2 \Big).
\end{displaymath}
The same goes for the trace norms $\norma{u}{\frac{1}{2}}{\partial \Omega}$ and the dual norms $\norma{u}{-1}{\Omega}$, $\norma{u}{-\frac{1}{2}}{\partial \Omega}$ and so forth. 
\end{remark}
We will sometimes use the following notation, for $h>0$:
\begin{displaymath}
\Omega_h  = \{ x \in \Omega \, \, \mathrm{such \, \, that } \, \, d(x, \partial \Omega) > h  \}.
\end{displaymath}

\subsection{A priori information.}
Here we present all the a priori hypotheses we will make all along the paper. \\
(1) {\it A priori information on the domain.}  \\
We assume $\Omega \subset \mathbb{R}^n$ to be a bounded domain, such that 
\begin{equation} \label{apriori0}
 \partial \Omega \text{  is connected,  } 
\end{equation}
and it has a sufficiently smooth boundary, i.e., 
\begin{equation} \label{apriori1}
\partial \Omega \text{ is of class } C^{2, \alpha} \text{ of constants } \rho_0, \, \, M_0, \end{equation} where $\alpha \in (0,1]$ is a real number, $M_0 > 0$, and $\rho_0 >0 $ is what we shall treat as our dimensional parameter. In what follows $\nu$ is the outer normal vector field to  $\partial \Omega$. We also require that 
\begin{equation} \label{apriori2} |\Omega| \le M_1 \rho_0^n, \end{equation} where $M_1 > 0$. \\
In our setup, we choose a special open and connected portion $\Gamma \subset \partial \Omega$ as being the accessible part of the boundary, where, ideally, all measurements are taken. We assume that there exists a point $P_0 \in \Gamma$ such that 
\begin{equation} \label{apriori2G}
\partial \Omega \cap B_{\rho_0}(P_0) \subset \Gamma.
\end{equation} 

(2) { \it A priori information about the obstacles.} \\
We consider $D \subset \Omega$, which represents the obstacle we want to detect from the boundary measurements, on which we require that 
\begin{equation} \label{apriori2bis}
\Omega \setminus \overline{D}  \text{ is connected, } 
\end{equation}
\begin{equation} \label{apriori2ter}
\partial D \text{ is connected. }
\end{equation}
We require the same regularity on $D$ as we did for $\Omega$, that is, 
\begin{equation} \label{apriori3} \partial D \text{ is of class } C^{2, \alpha} \text{ with constants }    \rho_0 , \, M_0.  \end{equation} 
In addition, we suppose that the obstacle is "well contained" in $\Omega$, meaning \begin{equation} \label{apriori4} d (D, \partial \Omega) \ge \rho_0. \end{equation}
\begin{remark}
We point out that, in principle, assumptions (\ref{apriori1}), (\ref{apriori3}) and (\ref{apriori4}) could hold for different values of $\rho_0$. If that were the case, it would be sufficient to redefine $\rho_0$ as the minimum among the three constants; then (\ref{apriori1}), (\ref{apriori2}) and (\ref{apriori3}) would still be true with the same $\rho_0$, while we would need to assume a different value of the constant $M_1$ in (\ref{apriori2}) accordingly. As a simple example, if $\Omega = B_1(0)$, and $D=B_{1/2}(0)$, then (\ref{apriori1}) is true for every $\rho_0 <1$, while (\ref{apriori3}) and (\ref{apriori4}) is true for all $\rho_0 <1/2$, so $\rho_0$ would be assumed to be less than $1/2$.  
\end{remark}
(3) { \it A priori information about the boundary data.} \\
For the Dirichlet-type data $g$ we assign on the accessible portion of the boundary $\Gamma$, we assume that
\begin{equation} \begin{split} \label{apriori5}
g \in \accan{\frac{3}{2}}{\partial \Omega}, \, \; \; g \not \equiv 0,  \\ \mathrm{supp} \,g \subset \subset \Gamma.
\end{split} \end{equation} 
As it is required in order to ensure the existence of a solution, we also require
\begin{equation} \label{aprioriexist} \int_{\partial \Omega}  g \, \mathrm{d} s =0.
\end{equation}
We also ask that, for a given constant $F>0$, we have 
\begin{equation} \label{apriori7}
\frac{\norma{g}{\frac{1}{2}}{\Gamma} }{\normadue{g}{\Gamma} } \le F.
\end{equation}
Under the above conditions on $g$, one can prove that there exists a constant $c>0$, only depending on $M_0$, such that the following equivalence relation holds: \begin{equation} \label{equivalence} 
\norma{g}{\frac{1}{2}}{\Gamma} \le \norma{g}{\frac{1}{2}}{\partial \Omega} \le c \norma{g}{\frac{1}{2}}{\Gamma}.
\end{equation}

\subsection{The main result.}
 Let $\Omega \subset \mathbb{R}^n$, and $\Gamma \subset \partial \Omega$ satisfy (\ref{apriori1})-(\ref{apriori2G}). Let $D_i \subset \Omega$, for $i=1,2$, satisfy (\ref{apriori2bis})-(\ref{apriori4}), and let us denote by $\Omega_i= \Omega \setminus \overline{D_i}$. We may state the main result as follows.
\begin{theorem}[Stability] \label{principale} Let $g \in \accan{\frac{3}{2}}{\Gamma}$ be the assigned boundary data, satisfying (\ref{apriori5})-(\ref{apriori7}). Let $u_i \in \accauno{\Omega_i}$ solve (\ref{NSE}) for $D=D_i$. If, for $\epsilon > 0 $, we have 
\begin{equation} \label{HpPiccolo}
\rho_0 \norma{\sigma(u_1, p_1)\cdot \nu  -\sigma(u_2,p_2) \cdot \nu }{-\frac{1}{2}}{\Gamma} \le \epsilon, \end{equation}
then 
\begin{equation}\label{stimstab} 
d_{\mathcal{H}} (\partial D_1, \partial D_2) \le \rho_0 \omega \Bigg( \frac{\epsilon}{\norma{g}{\frac{1}{2}}{\Gamma}}\Bigg),
\end{equation}
where $\omega : (0, +\infty) \to \mathbb{R}^+$ is an increasing function satisfying, for all $0<t<\frac{1}{e}$:
\begin{equation}
\omega(t) \le C (\log | \log t |)^{-\beta   }.
\end{equation}
The constants $C>0$ and $0<\beta<1$ only depend on $n$, $M_0$, $M_1$ and $F$.
\end{theorem}
\subsection{The Helmholtz-Weyl decomposition.}
We find it convenient to recall a classical result which will come in handy later on. A basic tool in the study of the Stokes equations (\ref{NSE}) is the Helmholtz-Weyl decomposition of the space $\elledue{\Omega}$ in two
orthogonal spaces:
\begin{equation}\label{HW} \elledue{\Omega} = H \oplus H^{\perp}, \end{equation}
where
\[ H =\{u \in \elledue{\Omega} \hspace{0.25em} : \dive u = 0, \hspace{0.25em}
   u|_{\partial \Omega} = 0\} \]
and
\[ H^{\perp} =\{u \in \elledue{\Omega} \hspace{0.25em} : \exists \hspace{0.25em} p \in \accan{1}{\Omega} \,: \;  u = \nabla p \hspace{0.25em} \}. \]
This decomposition is used, for example, to prove the existence of a solution of the Stokes system (among many others, see \cite{LadyK}).

From this, and using a quite standard "energy estimate" reasoning, one can prove the following (see \cite{LadyK} or \cite{Temam}, among many others):
\begin{theorem}[Regularity for the direct Stokes problem.] \label{TeoRegGen} 
Let $m \ge -1$ an integer number and let $E \subset \mathbb{R}^n$ be a bounded domain of class $C^r$ , with $r= \max \{ m+2, 2\}$. Let us consider the following problem:
\begin{equation}
 \label{NSEdiretto} \left\{ \begin{array}{rl}
    \dive \sigma(u,p) & = f \hspace{2em} \mathrm{\tmop{in}} \hspace{1em}
    E,\\
    \dive u & = 0 \hspace{2em} \mathrm{\tmop{in}} \hspace{1em} E, \\
    u & = g \hspace{2em} \mathrm{\tmop{on}} \hspace{1em} \partial    E,\\
  \end{array} \right.
\end{equation}
where $f \in \mathbf{H}^{m} (E)$ and $g \in \accan{m+\frac{3}{2}}{E}$. Then there exists a weak solution $(u,p) \in \mathbf{H}^{m+2}(E) \times H^{m+1} (E)$  and a constant $c_0$, only depending on the regularity constants of $E$ such that 
 \begin{equation}  \label{stimanormadiretto}
 \| u \|_{\mathbf{H}^{m+2} (E)} + \rho_0 \| p-p_E \|_{H^{m+1}(E)} \le   c_0  \big(\rho_0 \| f \|_{\mathbf{H}^{m} (E)} + \| g \|_{\mathbf{H}^{m+\frac{3}{2}}  (\partial E)} \big),
\end{equation}
where $p_E$ denotes the average of $p$ in $E$, $p_E = \frac{1}{|E|} \int_E p . $
\end{theorem}
Finally, we would like to recall the following version of Poincar\`e inequality, dealing with functions that vanish on an open portion of the boundary:
\begin{theorem}[Poincar\`e inequality.]
Let $E\subset \mathbb{R}^n$ be a bounded domain with boundary of Lipschitz class with constants $\rho_0$, $M_0$ and satisfying (\ref{apriori2}). Then for every $ u \in  \accan{1}{E}$ such that
\begin{displaymath} 
u = 0  \, \, \text{on}  \, \, \partial E \cap B_{\rho_0}(P), 
\end{displaymath}
where $P$ is some point in $\partial E$, we have
\begin{equation} \label{pancarre}
\normadue{u}{E} \le C \rho_0 \normadue{\nabla u}{E},
\end{equation}
where C is a positive constant only depending on $M_0$ and $M_1$.
\end{theorem}

\section{Proof of Theorem \ref{principale}.}
The proof of Theorem \ref{principale} relies on the following sequence of propositions.
\begin{proposition}[Lipschitz propagation of smallness] 
\label{teoPOS}
Let $E$ be a bounded Lipschitz domain with constants $\rho_0$, $M_0$, satisfying (\ref{apriori2}). 
Let $u$ be a solution to the following problem:
\begin{equation}
  \label{NSEPOS} \left\{ \begin{array}{rl}
    \dive\sigma(u,p) &= 0 \hspace{2em} \mathrm{\tmop{in}} \hspace{1em}
    E,\\
    \dive u & = 0 \hspace{2em} \mathrm{\tmop{in}} \hspace{1em} E,\\
    u & = g \hspace{2em} \mathrm{\tmop{on}} \hspace{1em} \partial E,\\
  \end{array} \right.
\end{equation}
where $g$ satisfies 
\begin{equation} \label{apriori5POS}
g \in \accan{\frac{3}{2}}{\partial E}, \, \; \; g \not \equiv 0, 
\end{equation} 
\begin{equation} \label{aprioriexistPOS} \int_{\partial E} g \,  \mathrm{d} s =0,
\end{equation}
\begin{equation} \label{apriori7POS}
\frac{\norma{g}{\frac{1}{2}}{\partial E} }{\normadue{g}{\partial E} } \le F,
\end{equation}
for a given constant $F>0$. Also suppose that there exists a point $P \in \partial E$ such that
\begin{equation} 
g = 0  \;\; \text{on}  \; \; \partial E \cap B_{\rho_0}(P).
\end{equation} Then there exists a constant $s>1$, depending only on $n$ and $M_0$ such that, for every $\rho >0$ and for every $\bar{x} \in E_{s\rho}$, we have 
\begin{equation} \label{POS}  \int_{B_{\rho}(\bar{x})} \! |\nabla u|^2 dx \ge C_\rho \int_{E} \! |\nabla u|^2 dx . \end{equation} 
Here $C_\rho>0$ is a constant depending only on  $n$, $M_0$, $M_1$, $F$, $\rho_0$ and $\rho$. The dependence of $C_\rho$ from $\rho$ and $\rho_0$ can be traced explicitly as 
\begin{equation} \label{crho} C_\rho = \frac{C}{\exp \Big[ A \big( \frac{\rho_0}{\rho}\big) ^B \Big] } \end{equation} where $A$, $B$, $C>0$ only depend on $n$, $M_0$, $M_1$ and $F$.
\end{proposition}
\begin{proposition}[Lipschitz propagation of smallness up to boundary data] \label{teoPOSC}
Under the hypotheses of Theorem \ref{principale}, for all $\rho>0$, if $\bar{x} \in (\Omega_i)_{{{(s+1)\rho}}}$,  we have for $i=1,2$:
\begin{equation} \label{POScauchy}
\frac{1}{\rho_0^{n-2}} \int_{B_{\rho}(\bar{x})} \! |\nabla u_i|^2 dx \ge  C_\rho \norma{g}{\frac{1}{2}}{\Gamma}^2,
\end{equation}
where $C_\rho$ is as in (\ref{crho}) (with possibly a different value of the term $C$), and $s$ is given by Proposition \ref{teoPOS}.
\end{proposition} 
\begin{proposition}[Stability estimate of continuation from Cauchy data] 
\label{teostabest} Under the hypotheses of Theorem \ref{principale} we have 
\begin{equation}\label{stabsti1} \frac{1}{\rho_0^{n-2}} \int_{D_2\setminus D_1} |\nabla u_1|^2 \le C \norma{g}{\frac{1}{2}}{\Gamma}^2 \omega\Bigg( \frac{\epsilon}{\norma{g}{\frac{1}{2}}{\Gamma}} \Bigg) \end{equation}
\begin{equation}\label{stabsti2} \frac{1}{\rho_0^{n-2}} \int_{D_1\setminus D_2} |\nabla u_2|^2 \le C \norma{g}{\frac{1}{2}}{\Gamma}^2 \omega\Bigg( \frac{\epsilon}{\norma{g}{\frac{1}{2}}{\Gamma}} \Bigg)\end{equation}
where $\omega$ is an increasing continuous function, defined on $\mathbb{R}^+$ and satisfying 
\begin{equation}\label{andomega} \omega(t) \le C \big( \log |\log t|\big)^{-c} \end{equation} 
for all $t < e^{-1}$, where $C$ only depends on $n$, $M_0$, $M_1$, $F$, and $c>0$ only depends on $n$. 
\end{proposition}
\begin{proposition}[Improved stability estimate of continuation] \label{teostabestimpr}
Let the hypotheses of Theorem \ref{principale} hold. Let $G$ be the connected component of $\Omega_1 \cap \Omega_2$ containing $\Gamma$, and assume that $\partial G$ is of Lipschitz class of constants $\tilde{\rho}_0$ and $\tilde{M_0}$, where $M_0>0$ and $0<\tilde{\rho}_0<\rho_0$. Then (\ref{stabsti1}) and (\ref{stabsti2}) both hold with $\omega$ given by
\begin{equation}\label{omegabetter}
\omega(t)= C |\log t|^{\gamma},
\end{equation}
defined for $t<1$, where $\gamma >0$ and $C>0$ only depend on $M_0$, $\tilde{M_0}$, $M_1$ and $\frac{\rho_0}{\tilde{\rho}_0}$. 
\end{proposition}
\begin{proposition} \label{teoreggra} Let $\Omega_1$ and $\Omega_2$ two bounded domains satisfying (\ref{apriori1}). Then there exist two positive numbers $d_0$, $\tilde{\rho}_0$, with $\tilde{\rho}_0 \le \rho_0$, such that the ratios $\frac{\rho_0}{\tilde{\rho}_0}$, $\frac{d_0}{\rho_0}$ only depend on $n$, $M_0$ 
and $\alpha$ such that, if 
\begin{equation} \label{relgr1}
d_{\mathcal{H}} (\overline{\Omega_1}, \overline{\Omega_2}) \le d_0,
\end{equation}
then there exists $\tilde{M}_0>0$ only depending on $n$, $M_0$ and $\alpha$ such that  every connected component of $\Omega_1 \cap \Omega_2$ has boundary of Lipschitz class with constants $\tilde{\rho}_0$, $\tilde{M}_0$. 
\end{proposition}
We postpone the proofs of Propositions \ref{teoPOS} and \ref{teoPOSC}  to Section 4, while Propositions \ref{teostabest} and \ref{teostabestimpr} will be proven in Section 5. The proof of Proposition \ref{teoreggra} is purely geometrical and can be found in \cite{ABRV}. 
\begin{proof}[Proof of Theorem \ref{principale}.]
Let us call 
\begin{equation} \label{distanza} d= d_\mathcal{H}(\partial D_1, \partial D_2). \end{equation}
Let $\eta$ be the quantity on the right hand side of (\ref{stabsti1}) and (\ref{stabsti2}), so that
\begin{equation} \label{eta} \begin{split}
\int_{D_2 \setminus D_1} |\nabla u_1|^2 \le \eta, \\
\int_{D_1 \setminus D_2} |\nabla u_2|^2 \le \eta. \\
\end{split}\end{equation}
We can assume without loss of generality that there exists a point $x_1 \in \partial D_1$ such that dist$(x_1, \partial D_2)=d$. That being the case, we distinguish two possible situations: \\ (i) $B_d(x_1) \subset D_2$, \\ (ii) $B_d(x_1) \cap D_2 =\emptyset$.\\
In case (i), by the regularity assumptions on $\partial D_1$, we find a point $x_2 \in D_2 \setminus D_1$ such that $B_{td}(x_2) \subset D_2 \setminus D_1$, where $t$ is small enough (for example, $t=\frac{1}{1+\sqrt{1+M_0^2}}$ suffices). Using (\ref{POScauchy}), with $\rho = \frac{t d}{s}$ we have 
\begin{equation} \label{stimapos} \int_{B_\rho (x_2) } |\nabla u_1|^2 dx \ge   \frac{C\rho_0^{n-2}}{\exp \Big[A \big(\frac{s\rho_0}{t d }\big)^B\Big]}  \norma{g}{\frac{1}{2}}{\Gamma}^2.  
\end{equation}
By Proposition \ref{teostabest}, we have: 
\begin{equation} \label{quellaconomega}
\omega\Bigg( \frac{\epsilon}{ \norma{g}{\frac{1}{2}}{\Gamma}} \Bigg)   \ge \frac{C}{\exp \Big[A \big(\frac{s\rho_0}{t d }\big)^B\Big]}  , \end{equation}  
and solving for $d$ we obtain an estimate of log-log-log type stability:
\begin{equation}\label{logloglog}
d \le C \rho_0 \Bigg\{ \log \Bigg[ \log \Bigg|\log\frac{\epsilon}{\norma{g}{\frac{1}{2}}{\Gamma}} \Bigg| \Bigg] \Bigg\}^{-\frac{1}{B}},
\end{equation} 
provided $\epsilon < e^{-e} \norma{g}{\frac{1}{2}}{\Gamma}$: this is not restrictive since, for larger values of $\epsilon$, the thesis is trivial. If we call $d_0$ the right hand side of (\ref{logloglog}), we have that there exists $\epsilon_0$ only depending on $n$, $M_0$, $M_1$ and $F$ such that, if $\epsilon \le \epsilon_0$ then $d\le d_0$. Proposition \ref{teoreggra} then applies, so that $G$ satisfies the hypotheses of Proposition \ref{teostabestimpr}. This means that we may choose  $\omega$ of the form (\ref{omegabetter}) in (\ref{quellaconomega}), obtaining (\ref{stabsti1}).
Case (ii) can be treated analogously, upon substituting $u_1$ with $u_2$.
\end{proof}

\section{Proof of Proposition \ref{teoPOS}.}
The main idea of the proof of Proposition \ref{teoPOS} is a repeated application of a three-spheres type inequality. Inequalities as such play a crucial role in almost all stability estimates from Cauchy data, thus they have been adapted to a variety of elliptic PDEs: in the context of the scalar elliptic equations (see \cite{ABRV}), then in the determination of cavities or inclusions in elastic bodies (\cite{MR}, \cite{MRC}) and more in general, for scalar elliptic equations (\cite{ARRV}) as well as systems (\cite{LNW}) with suitably smooth coefficients. We recall in particular the following estimate, which is a special case of a result of Nagayasu, Lin and Wang (\cite{LNW}), dealing with systems of differential inequalities of the form:
\begin{equation} \label{ellgeneral} |\triangle^l u^i | \le K_0 \sum_{|\alpha| \le \big[ \frac{3l}{2} \big] }  | D^\alpha u |  \,  \quad i=1,\dots , n. \end{equation}
Then the following holds (see \cite{LNW}):
\begin{theorem}[Three spheres inequality.] \label{teotresfere} Let $E \subset \mathbb{R}^n$  be a bounded domain  with Lipschitz boundary with constants $\rho_0$, $M_0$. Let $B_R(x)$ a ball contained in $E$, and let $u \in \accan{2l}{E}$ be a solution to (\ref{ellgeneral}). Then there exists a real number $\vartheta^* \in (0, e^{-1/2})$, depending only on $n$, $l$ and $K_0$ such that, for all $0<r_1 <r_2 <\vartheta^* r_3$ with $r_3  \le R$ we have:
\begin{equation} \label{tresfere}  \int_{B_{r_2}} \! | u|^2 dx \le C \Big(\int_{B_{r_1} } \! | u|^2 dx \Big)^\delta  \Big(\int_{B_{r_3}} \! | u|^2 dx \Big)^{1-\delta} \end{equation} 
where $\delta \in (0,1)$ and $C>0$ are constants depending only on $n$, $l$, $K_0$, $\frac{r_1}{r_3}$ and $\frac{r_2}{r_3}$, and the balls $B_{r_i}$ are centered in $x$.
\end{theorem}
First, we show that Proposition \ref{teoPOSC} follows from Proposition \ref{teoPOS}:
\begin{proof}[Proof of Proposition \ref{teoPOSC}.]
From Proposition \ref{teoPOS} we know that \begin{displaymath}  \int_{B_{\rho}(x)} \! |\nabla u_i|^2 dx \ge C_\rho  \int_{\Omega \setminus \overline{D_i}} \! |\nabla u_i|^2 dx, \end{displaymath}
where $C_\rho$ is given in (\ref{crho}).
We have, using Poincar\`e inequality (\ref{pancarre}) and the trace theorem, 
\begin{equation} \label{altofrequenza} \begin{split}  \int_{\Omega\setminus \overline{D_i}}  |\nabla u_i|^2 dx \ge C \rho_0^{n-2} \norma{u_i}{1}{\Omega \setminus \overline{D_i}}^2  \ge  C \rho_0^{n-2} \norma{g}{\frac{1}{2}}{\partial \Omega}^2. \end{split}\end{equation}
Applying the above estimate to (\ref{POS}) and using (\ref{equivalence}) will prove our statement.
\end{proof}
Next, we introduce a lemma we shall need later on:
\begin{lemma} \label{42}
Let the hypotheses of Proposition \ref{teoPOS} be satisfied. Then 
\begin{equation}
\normadue{u}{E} \ge \frac{C}{F^2} \rho_0 \normadue{\nabla u}{E}
\end{equation}
where $C>0$ only depends on $n$, $M_0$ and $M_1$.
\end{lemma}
The proof is obtained in \cite{MRC}, with minor modifications. We report it here for the sake of completeness.
\begin{proof}
Assume $\rho_0=1$, otherwise the thesis follows by scaling. The following trace inequality holds (see \cite[Theorem 1.5.1.10]{21}):
\begin{equation} \label{trace1} 
\normadue{u}{\partial E} \le C (\normadue{\nabla u}{E} \normadue{u}{E} + \normadue{u}{E}^2), 
\end{equation}
where $C$ only depends on $M_0$ and $M_1$. Using the Poincar\`e inequality (\ref{pancarre}), we have 
\begin{equation}
\frac{\normadue{\nabla u}{E} }{ \normadue{u}{E} } \le C \frac{\normadue{\nabla u}{E}^2}{\normadue{u}{\partial E}^2}.
\end{equation}
This, together with (\ref{stimanormadiretto}), immediately gives the thesis. 
\end{proof}
A proof of Proposition \ref{teoPOS} has already been obtained in \cite{MRC} dealing with linearized elasticity equations; we give a sketch of it here, with the due adaptations.
\begin{proof}[Proof of Proposition \ref{teoPOS}.]
We outline the main steps taken in the proof. First, we show that the three spheres inequality (\ref{tresfere}) applies to $\nabla u$. Then, the goal is to estimate $\normadue{\nabla u}{E}$ by covering the set $E$ with a sequence of cubes $Q_i$ with center $q_i$ of "relatively small" size. Each of these cubes is contained in a sphere $S_i$, thus we estimate the norm of $\nabla u$ in every sphere of center $q_i$, by connecting $q_i$ with $x$ with a continuous arc, and apply an iteration of the three spheres inequality to estimate $\normadue{\nabla u}{S_i}$ in terms of $\normadue{\nabla u}{B_\rho(x)}$. However, the estimates deteriorate exponentially as we increase the number of spheres  (or equivalently, if the radius $\rho$ is comparable with the distance of $x$ from the boundary) giving an exponentially worse estimate of the constant $C_\rho$. To solve this problem, the idea is to distinguish two areas within $E_{s \rho}$, which we shall call $A_1$, $A_2$. We consider $A_1$ as the set of points $y \in E_{s \rho}$ such that $\mathrm{dist}(y, \partial E)$ is sufficiently large, whereas  $A_2$ is given as the  complement in $E_{s \rho}$ of $A_1$. Then, whenever we need to compare the norm of $\nabla u$ on two balls whose centers lie in $A_2$, we reduce the number of spheres by iterating the three spheres inequality over a sequence of balls with increasing radius, exploiting the Lipschitz character of $\partial E$ by building a cone to which all the balls are internall tangent to. Once we have reached a sufficiently large distance from the boundary, we are able to pick a chain of larger balls, on which we can iterate the three speres inequality again without deteriorating the estimate too much. This line of reasoning allows us to estimate the norm of $\nabla u$ on any sphere contained in $E_{s \rho}$, thus the whole $\normadue{\nabla u}{E}$. \\

{\bf Step 1.}
{ \it If $u \in \accauno{E}$ solves (\ref{NSEPOS}) then the three spheres inequality (\ref{tresfere}) applies to $ \nabla u$.}  
\begin{proof}[Proof of Step 1.] We show that $u$ can be written as a solution of a system of the form (\ref{ellgeneral}). By Theorem \ref{TeoRegGen}, we have $u \in \mathbf{H}^2(E)$ so that we may take the laplacian of the second equation in (\ref{NSE}):
\begin{displaymath} 
\triangle \dive u  =0. \end{displaymath}
Commuting the differential operators, and recalling the first equation in (\ref{NSE}),
\begin{displaymath} 
\triangle{p}=0
\end{displaymath}
thus $p$ is harmonic, which means that, if we take the laplacian of the first equation in (\ref{NSE}) we get 
\begin{displaymath}
\triangle^2 u=0,
\end{displaymath}
so that $\nabla u$ is also biharmonic, hence the thesis. 
\end{proof}
In what follows, we will always suppose $\rho_0=1$: The general case is treated by a rescaling argument on the biharmonic equation.
We closely follow the geometric construction given in \cite{MRC}. In the aforementioned work the object was to estimate $\| \hat{\nabla} u\|$, by applying the three spheres inequality to $\hat{\nabla} u$ (the symmetrized gradient of $u$); in order to relate it to the boundary data, this step had to be combined with Korn and Caccioppoli type inequalities. Here the estimates are obtained for $\|\nabla u \|$. \\
From now on we will denote, for $z \in \mathbb{R}^n$, $\xi \in \mathbb{R}^n$ such that $|\xi|=1$, and $\vartheta >0$, 
\begin{equation} \label{cono}
C(z, \xi, \vartheta)= \Big\{ x \in \mathbb{R}^n \text{ s.t. } \frac{(x-z) \cdot \xi}{|x-z|} > \cos \vartheta    \Big\}
\end{equation}
the cone of vertex $z$, direction $\xi$ and width $2 \vartheta$. \\ Exploiting the Lipschitz character of $\partial E$, we can find $\vartheta_0 >0$ depending only on $M_0$, $\vartheta_1>0$,  $\chi >1$ and $s>1$ depending only on $M_0$ and $n$, such that the following holds (we refer to \cite{MRC} for the explicit expressions of the constants $\vartheta_0$, $\vartheta_1$,  $\chi$, $s$,  and for all the detailed geometric constructions).\\
 {\bf Step 2.} 
{ \it Choose $0<\vartheta^* \le 1$ according to Theorem \ref{teotresfere} .There exists $\overline{\rho}>0$, only depending on $M_0$, $M_1$ and $F$, such that: \\
If $ 0<\rho \le \bar{\rho}$, and $x \in E$ is such that $ s \rho < \mathrm{dist} (x, \partial E) \le \frac{\vartheta^*}{4}$, then there exists $\hat{x} \in E$ satisfying the following conditions:  
\begin{enumerate}
\item[(i)] $B_{\frac{5 \chi \rho}{\vartheta^*}} (x) \subset C(\hat{x},e_n=\frac{x-\hat{x}}{|x-\hat{x}|} , \vartheta_0) \cap B_{\frac{\vartheta^* }{8}}
(\hat{x}) \subset E$,
\item[(ii)] Let $x_2 = x+ \rho(\chi+1)e_n$. Then the balls $B_\rho (x)$ and $B_{\chi \rho} (x_2)$ are internally tangent to the cone $C(\hat{x},e_n, \vartheta_1)$.
\end{enumerate}}
The idea is now to repeat iteratively the construction made once in Step 2. We define the following sequence of points and radii:
\begin{displaymath} \begin{split}
\rho_1 &= \rho, \; \; \; \rho_k = \chi \rho_{k-1}, \; \; \text{  for  } k \ge 2, \\ 
x_1 &= x, \; \; \; x_k=x_{k-1}+ (\rho_{k-1} + \rho_k) e_n , \qquad \text{  for  } k \ge 2. \end{split}
\end{displaymath}
We claim the following geometrical facts (the proof of which can be found again in \cite{MRC}, except the first, which is \cite[Proposition 5.5]{ARRV}): \\ 

{\it There exist $0<h_0<1/4$ only depending on $M_0$,  $\bar{\rho} >0$ only depending on $M_0$, $M_1$ and $F$,  an integer $k(\rho)$ depending also on $M_0$ and $n$,  such that, for all $h \le h_0$,  $0<\rho \le \bar{\rho}$ and for all integers $1<k \le k(\rho)-1$ we have: \begin{enumerate}
\item \label{fatto0} $E_h$ is connected,
\item \label{fatto1} $B_{\rho_k}(x_k)$ is internally tangent to $C(\hat{x}, e_n, \vartheta_1) $,
\item \label{fatto2}  $B_{\frac{5 \chi \rho_k}{\vartheta^*}}(x_k) $ is internally tangent to $C(\hat{x}, e_n, \vartheta_0) $, 
 \item The following inclusion holds: \begin{equation} \label{fatto3}  B_{\frac{5 \rho_k}{\vartheta^*}}(x_k) \subset B_{\frac{\vartheta^*}{8}}(\hat{x}), \end{equation}
\item $k(\rho)$ can be bounded from above as follows: \begin{equation} \label{432} k(\rho) \le \log \frac{\vartheta^* h_0}{5 \rho} +1.  \end{equation}
\end{enumerate}
}
Call $\rho_{k(\rho)}= \chi^{k(\rho)-1} \rho$; from (\ref{432}) we have that 
\begin{equation} \label{433}
\rho_{k(\rho)} \le \frac{\vartheta^* h_0}{5}.
\end{equation}
In what follows, in order to ease the notation, norms will be always understood as being $\mathbf{L}^2$ norms, so that $\|\cdot  \|_U$ will stand for $\normadue{\cdot}{U}$. \\
{\bf Step 3.} {\it For all $0<\rho \le \bar{\rho}$ and for all $x \in E$ such that $s \rho \le \mathrm{dist}(x, \partial E) \le \frac{\vartheta^*}{4}$, the following hold:
\begin{equation} \label{434}
\frac{\nor{\nabla u}{B_{\rho_{k(\rho)}} (x_{k(\rho)}) }}{\nor{\nabla u}{E}}
\le C \Bigg( \frac{\nor{\nabla u}{B_{\rho} (x)}}{\nor{\nabla u}{E}} \Bigg)^{\delta_\chi^{k(\rho)-1}},
\end{equation} 
\begin{equation} \label{435}
\frac{\nor{\nabla u}{B_{\rho} (x)} }{\nor{\nabla u}{E}}
\le C \Bigg( \frac{\nor{\nabla u}{B_{\rho_{k(\rho)}} (x_{\rho_{k(\rho)}})}}{\nor{\nabla u}{E}} \Bigg)^{\delta^{k(\rho)-1}},
\end{equation} 
where $C>0$ and $0<\delta_\chi<\delta<1$ only depend on $M_0$. 
}
\begin{proof}[Proof of Step 3.]
We apply to $\nabla u$ the three-spheres inequality, with balls of center $x_j$ and radii $r_1^{j}=\rho_j$, $r_2^{j}=3\chi \rho_j$, $r_3^{j}=4 \chi\rho_j$, for all $j=1, \dots, k(\rho)-1$. Since 
$B_{r_1^{j+1}}(x_{j+1}) \subset B_{r_2^j}(x_j)$, 
by the three spheres inequality, there exists $C$ and $\delta_\chi$ only depending on $M_0$, such that:
\begin{equation} \label{437}  \nor{\nabla u}{B_{\rho_{j+1}}(x_{j+1})}  \le C \Big(\nor{\nabla u}{B_{\rho_j}(x_j) }\Big)^{\delta_\chi}  \Big(\nor{\nabla u}{B_{4\chi\rho_j}(x_j)}\Big) ^{1-\delta_\chi}. \end{equation}
This, in turn, leads to: 
\begin{equation} \label{438} \frac{\nor{ \nabla u}{B_{\rho_{j+1}}(x_{j+1})} }{\nor{\nabla u}{E}} \le C \Bigg(\frac{\nor{ \nabla u}{B_{\rho_j}(x_j)} }{\nor{\nabla u}{E}}  \Bigg)^{\delta_\chi},\end{equation} 
for all $j=0, \dots k(\rho)-1$. 
Now call 
\begin{displaymath} 
m_k = \frac{\nor{ \nabla u}{B_{\rho_{j+1}}(x_{j+1})} }{\nor{\nabla u}{E}}. 
\end{displaymath}
so that (\ref{438}) reads
\begin{equation} \label{stepdue}
m_{k+1} \le C m_k^{\delta_\chi} \,  \nor{\nabla u}{E}^{1-\delta_\chi} , 
\end{equation}
which, inductively, leads to \begin{equation} \label{steptre}
m_{N} \le \tilde{C} m_0^\alpha, 
\end{equation}
where $\tilde{C} = C^{1+\delta_\chi+ \dots+ \delta_\chi^{k(\rho)-2}}$. Since $ 0<\delta_\chi <1$, we have $1+\delta_\chi+ \dots+ \delta_\chi^{k(\rho)-2} \le \frac{1}{1-\delta_\chi}$, and since we may take $C>1$, 
\begin{equation} \label{stepquattro} 
\tilde{C} \le C^{\frac{1}{1-\delta_\chi}}.
\end{equation}
 Similarly, we obtain (\ref{435}): we find a $0<\delta<1$ such that the three spheres inequality applies to the balls $B_{\rho_j}(x_j)$, $B_{3\rho_j}(x_j)$ $B_{4\rho_j}(x_j)$ for $j=2,\dots, k(\rho)$; observing that $B_{\rho_{j}(x_{j-1})} \subset B_{3\rho_j}(x_j)$, the line of reasoning followed above applies identically. 
\end{proof}
{\bf Step 4.}  
\\{\it For all $0<\rho \le \overline{\rho}$, and for every $\bar{x} \in E_{s\rho}$  we have   \begin{equation}\label{453}
\frac{\nor{\nabla u}{B_{\rho}(y)}}{\nor{\nabla u}{E}}
\le C \Bigg( \frac{\nor{\nabla u}{B_\rho (\bar{x})}}{\nor{\nabla u}{E}} \Bigg)^{ \delta_\chi^{A+B\log \frac{1}{\rho}}} .
\end{equation} } 
\begin{proof} We distinguish two subcases:
\begin{enumerate}
 \item[\it (i).] $\bar{x}$ is such that $\mathrm{dist} (\bar{x}, \partial E) \le \frac{\vartheta^*}{4}$, 
\item[\it (ii).] $\bar{x}$ is such that $\mathrm{dist}(\bar{x}, \partial E) > \frac{\vartheta^*}{4}$.
\end{enumerate}
 {\it Proof of Case (i).}
Let us consider $\delta$, $\delta_\chi$ we introduced in Step 3
. Take any point $y \in E$ such that $s \rho < \mathrm{dist}(y, \partial E) \le \frac{\vartheta^*}{4}$.
By construction, the set $E_{\frac{5\rho_{k(\rho)}}{\vartheta^*}}$ is connected, thus there exists a continuous path $\gamma : [0,1] \to E_{\frac{5\rho_{k(\rho)}}{\vartheta^*}}$ joining $\bar{x}_{k(\rho)}$ to $y_{k(\rho)}$. We define a ordered sequence of times $t_j$, and a corresponding sequence of points $x_j= \gamma(t_j)$, for $j=1, \dots, L$ in the following way: $t_1=0$, $t_L =1$, and
\begin{displaymath} 
t_j= \mathrm{max} \{t\in (0,1]  \text{  such that  } |\gamma(t)- x_i| = 2 \rho_{k(\rho)} \} \; \text{, if } |x_i-y_{k(\rho)}| > 2 \rho_{k(\rho)}, 
\end{displaymath}
otherwise, let $k=L$ and the process is stopped. Now, all the balls $B_{\rho_{k(\rho)}}(x_i)$ are pairwise disjoint, the distance between centers $| x_{j+1}-x_j | = 2 \rho_{k(\rho)}$ for all $j=1 \dots L-1$ and for the last point, $|x_L - y_{k(\rho)}| \le 2 \rho_{k(\rho)}$. The number of points, using (\ref{apriori2}), is at most
\begin{equation} \label{sferealmassimo} L \le \frac{M_1}{\omega_n \rho_{k(\rho)}^n}.  \end{equation} 
Iterating the three spheres inequality over this chain of balls, we obtain 
\begin{equation} \label{442}
\frac{\nor{\nabla u}{B_{\rho_{k(\rho)}}(y_{k(\rho)})}}{\nor{\nabla u}{E}}  \le C \Bigg( \frac{\nor{\nabla u}{B_{\rho_{k(\rho)}}(\bar{x}_{k(\rho)}) }}{\nor{\nabla u}{E}} \Bigg)^{\delta^L}
\end{equation}
On the other hand, by the previous step we have, applying (\ref{434}) and (\ref{435}) for $x=\bar{x}$ and $x=y$ respectively, 
\begin{equation} \label{443}
\frac{\nor{\nabla u}{B_{\rho_{k(\rho)}} (\bar{x}_{k(\rho)}) }}{\nor{\nabla u}{E}}
\le C \Bigg( \frac{\nor{\nabla u}{B_{\rho} (\bar{x})}}{\nor{\nabla u}{E}} \Bigg)^{\delta_\chi^{k(\rho)-1}},
\end{equation}
\begin{equation} \label{444}
\frac{\nor{\nabla u}{B_{\rho}(y)} }{\nor{\nabla u}{E}}
\le C \Bigg( \frac{\nor{\nabla u}{B_{\rho_{k(\rho)}} (y_{k(\rho)})}}{\nor{\nabla u}{E}} \Bigg)^{\delta^{k(\rho)-1}},
\end{equation}
where $C$, as before, only depends on $n$ and $M_0$. Combining (\ref{442}), (\ref{443}) and (\ref{444}), we have 
\begin{equation} \label{445}
\frac{\nor{\nabla u}{B_{\rho}(y)} }{\nor{\nabla u}{E}}
\le C \Bigg( \frac{\nor{\nabla u}{B_{\rho} (\bar{x})}}{\nor{\nabla u}{E}} \Bigg)^{\delta_\chi^{k(\rho)-1} \delta^{k(\rho)+L-1}},
\end{equation}
for every $y \in E_{s\rho}$ satisfying $\mathrm{dist} (y, \partial E) \le \frac{\vartheta^*}{4}$. Now consider $y \in E$ such that $\mathrm{dist} (y, \partial E) > \frac{\vartheta^*}{4}$.
Call \begin{equation} \label{446}
\tilde{r}= \vartheta^* \rho_{k(\rho)}.
\end{equation}
By construction (\ref{433}) and (\ref{fatto3}) we have 
\begin{equation}
\mathrm{dist}(\bar{x}_{k(\rho)}, \partial E) \ge \frac{5 \rho_{k(\rho)}}{\vartheta^*} > \frac{5}{\vartheta^*} \tilde{r} ,
\end{equation}
\begin{equation}
\mathrm{dist}(y, \partial E) \ge \frac{5 \rho_{k(\rho)}}{\vartheta^*} > \frac{5}{\vartheta^*} \tilde{r},
\end{equation}
and again $E_{\frac{5}{\vartheta^*} \tilde{r}}$ is connected, since $\tilde{r}< \rho_{k(\rho)}$.  We are then allowed to join $\bar{x}_{k(\rho)}$ to $y$ with a continuous arc, and copy the argument seen before over a chain of at most $\tilde{L}$ balls of centers $x_j \in E_{\frac{5}{\vartheta^*} \tilde{r}}$ and radii $\tilde{r}$, $3\tilde{r}$, $4\tilde{r}$, where \begin{equation} \label{sferealmassimotilde} \tilde{L} \le \frac{M_1}{\omega_n \tilde{r}^n}.  \end{equation} 
Up to possibly shrinking $\overline{\rho}$, we may suppose $\rho \le \tilde{r}$; iterating the three spheres inequality as we did before, we get
\begin{equation}
\label{451}
\frac{\nor{\nabla u}{B_{\tilde{r}}(y)} }{\nor{\nabla u}{E}}
\le C \Bigg( \frac{\nor{\nabla u}{B_{\tilde{r}} (\bar{x}_{k(\rho)})}}{\nor{\nabla u}{E}} \Bigg)^{ \delta^{\tilde{L}}},
\end{equation}
which, in turn, by (\ref{443}) and since $\rho \le \tilde{r} < \rho_{k(\rho)}$,  becomes
\begin{equation} \label{452}
\frac{\nor{\nabla u}{B_{\rho}(y)}}{\nor{\nabla u}{E}}
\le C \Bigg( \frac{\nor{\nabla u}{B_\rho (\bar{x})}}{\nor{\nabla u}{E}} \Bigg)^{ \delta_\chi^{k(\rho)-1}\delta^{\tilde{L}}},
\end{equation}
with $C$ depending only on $M_0$ and $n$. The estimate (\ref{452}) holds for all $y \in E$ such that $\mathrm{dist} (y, \partial E) > \frac{\vartheta^*}{4}$. 
We now put (\ref{432}), (\ref{452}), (\ref{445}), (\ref{sferealmassimo})  (\ref{sferealmassimotilde}) together, by also observing that $\delta_\chi \le \delta$ and trivially $\frac{\nor{\nabla u}{B_{\rho}(y)}}{\nor{\nabla u}{E}} \le 1$, we obtain precisely (\ref{453}), for $\rho \le \overline{\rho}$, where $C>1$ and $B>0$ only depend on $M_0$, while $A>0$ only depend on $M_0$ and $M_1$.\\
{ \it Proof of Case (ii).} 
We use the same constants $\delta$ and $\delta_\chi$ introduced in Step 3. Take $\rho \le \bar{\rho}$, then  $B_{s\rho}(\bar{x}) \subset B_{\frac{\vartheta^*}{16}}(\bar{x})$, and for any point $\tilde{x}$ such that $|\bar{x} - \tilde{x}| = s \rho$, we have $B_{\frac{\vartheta^*}{8}}(\tilde{x}) \subset E$. Following the construction made in Steps 2 and 3, we choose a point $\bar{x}_{k(\rho)} \in E_{\frac{5}{\vartheta^*}\rho_{k(\rho)}}$, such that 
\begin{equation}
\frac{\nor{\nabla u}{B_{\rho_{k(\rho)}}(\bar{x}_{k(\rho)})}}{\nor{\nabla u}{E}}
\le C \Bigg( \frac{\nor{\nabla u}{B_\rho (\bar{x})}}{\nor{\nabla u}{E}} \Bigg)^{ \delta_\chi^{k(\rho)-1}}, \end{equation}
with $C>1$ only depending on $n$, $M_0$. 
If $y \in E$ is such that $s\rho <\mathrm{dist} (y, \partial E) \le \frac{\vartheta^*}{4}$, then, by the same reasoning as in Step 4.(i), we obtain 
\begin{equation}\label{459}
\frac{\nor{\nabla u}{B_{\rho}(y)}}{\nor{\nabla u}{E}}
\le C \Bigg( \frac{\nor{\nabla u}{B_\rho (\bar{x})}}{\nor{\nabla u}{E}} \Bigg)^{ \delta_\chi^{k(\rho)-1} \delta^{k(\rho)+L-1}}, \end{equation}
with $C>1$ again depending only on $M_0$. 
If, on the other hand, $y \in E$ is such that $\mathrm{dist}(y, \partial E) \ge \frac{\vartheta^*}{4}$, taking $\tilde{r}$ as in (\ref{446}), using the same argument as in Step 4.(i), we obtain 
\begin{equation}\label{460}
\frac{\nor{\nabla u}{B_{\rho}(y)}}{\nor{\nabla u}{E}}
\le C \Bigg( \frac{\nor{\nabla u}{B_\rho(\bar{x})}}{\nor{\nabla u}{E}} \Bigg)
^{\delta_\chi^{k(\rho)-1} \delta^{\tilde{L}}},
\end{equation}
where again $C>1$ only depends on $M_0$. 
From (\ref{459}),(\ref{460}), (\ref{sferealmassimo}),(\ref{sferealmassimotilde}) and (\ref{432}), and recalling that, again, $\delta_\chi \le \delta$, and $\frac{\nor{\nabla u}{B_{\rho}(y)}}{\nor{\nabla u}{E}} \le 1$, we obtain 
\begin{equation}
\frac{\nor{\nabla u}{B_{\rho}(y)}}{\nor{\nabla u}{E}}
\le C \Bigg( \frac{\nor{\nabla u}{B_\rho (\bar{x})}}{\nor{\nabla u}{E}} \Bigg)^{ \delta_\chi^{A+B\log\frac{1}{\rho}} },
\end{equation}
where $C>1$ and $B>0$ only depend on  $M_0$, while $A>0$ only depends on  $M_0$, $M_1$. 
\end{proof}

{\bf Step 5.} {\it For every $\rho \le \bar{\rho}$ and for every $\bar{x} \in E_{s\rho}$ the thesis (\ref{POS}) holds. }
\begin{proof}[Proof of Step 5] 
Suppose at first that $\bar{x} \in E_{s\rho}$ satisfies $\mathrm{dist} (\bar{x}, \partial E) \le \frac{\vartheta^*}{4}$. 
We cover $E_{(s+1)\rho}$ with a sequence of non-overlapping cubes of side $l= \frac{2 \rho}{\sqrt{n}}$, so that every cube is contained in a ball of radius $\rho$ and center in $E_{s \rho}$. The number of cubes is bounded by
\begin{displaymath}
N= \frac{|\Omega|n^{\frac{n}{2}}}{(2\rho)^n} \le \frac{M_1 n^{\frac{n}{2}}}{(2 \rho)^n}.
\end{displaymath}
If we then sum over $k=0$ to $N$ in (\ref{453}) we can write:
\begin{equation} \label{stepcinque}
\frac{\nor{\nabla u}{E_{(s+1) \rho}}}{\nor{\nabla u}{E}} \le C \rho^{-\frac{n}{2}} \Biggr( \frac{\nor{\nabla u}{B_\rho(\bar{x})} }{\nor{\nabla u}{E}} \Biggr)^{\delta_\chi^{A+B\log \frac{1}{\rho}}} . 
\end{equation}
Here $C$ depends only on $M_0$.
Now, we need to estimate the left hand side in (\ref{stepcinque}). In order to do so, we start by writing 
\begin{equation} \label{unomeno} \frac{\nor{\nabla u}{E_{(s+1) \rho}}}{\nor{\nabla u}{E}} =1-\frac{\nor{\nabla u}{E \setminus E_{(s+1) \rho}}}{\nor{\nabla u}{E}}.
\end{equation}
By Lemma \ref{42} and the H\"older inequality,
\begin{equation} \label{buttatali} 
\nor{\nabla u}{E \setminus E_{(s+1)\rho}}^2 \le C F^2  \nor{u}{E \setminus E _{(s+1)\rho}}^2 \le C F^2 |E \setminus E _{(s+1)\rho}| ^{\frac{1}{n}} \| u\|^2_{\mathbf{L}^{\frac{2n}{n-1}}(E \setminus E _{(s+1)\rho})}.
\end{equation}
On the other hand, by the Sobolev  and the Poincar\`e inequalities:
\begin{equation} \label{buttatali2}
 \| u\|_{\mathbf{L}^{\frac{2n}{n-1}}(E 
)} \le C \norma{ u}{\frac{1}{2}}{E 
} \le C \nor{u}{E} \le C \nor{\nabla u} {E}.
\end{equation}
It can be proven (see \cite[Lemma 5.7]{ARRV}) that
\begin{equation} \label{buttatali3}
|E \setminus E_{(s+1)\rho }| \le C \rho,
\end{equation}
where $C$ depends on $M_0$, $M_1$ and $n$. We thus obtain that 
\begin{equation} \label{storysofar}
\frac{\nor{\nabla u}{E \setminus E_{(s+1)\rho}}}{\nor{\nabla u} {E}} \le C F^2 |E \setminus E_{(s+1)\rho}|^{\frac{1}{n}}.
\end{equation}
Therefore, combining
(\ref{storysofar}) and (\ref{buttatali3}), we have that for $\rho \le \bar{\rho}$, 
\begin{equation} \label{unmezzo} 
\frac{\nor{\nabla u}{E _{(s+1) \rho}}}{\nor{\nabla u}{E}} \le \frac{1}{2}, 
\end{equation}
which, inserted into (\ref{stepcinque}) yields
\begin{equation*}
\int_{B_\rho(\bar{x})} |\nabla u|^2 \ge C \rho^{n\delta_\chi^{-A-B\log\frac{1}{\rho}}} \int_E |\nabla u|^2.
\end{equation*}
Since for all $t>0$ we have $|\log t| \le \frac{1}{t}$, it is immediate to verify that (\ref{POS}) holds.
Now take $\bar{x} \in E_{s\rho}$ such that $\mathrm{dist}(\bar{x}, \partial E) > \frac{\vartheta^*}{4}$.  Then  $B_{s\rho}(\bar{x}) \subset B_{\frac{\vartheta^*}{16}}(\bar{x})$, then for any point $\tilde{x}$ such that $|\bar{x} - \tilde{x}| = s \rho$, we have $B_{\frac{\vartheta^*}{8}}(\tilde{x}) \subset E$. Following the construction made in Steps 2 and 3, we choose a point $\bar{x}_{k(\rho)} \in E_{\frac{5}{\vartheta^*}\rho_{k(\rho)}}$, such that 
\begin{equation}
\frac{\nor{\nabla u}{B_{\rho_{k(\rho)}}(\bar{x}_{k(\rho)})}}{\nor{\nabla u}{E}}
\le C \Bigg( \frac{\nor{\nabla u}{B_\rho (\bar{x})}}{\nor{\nabla u}{E}} \Bigg)^{ \delta_\chi^{k(\rho)-1}}, \end{equation}
with $C>1$ only depends on $n$, $M_0$. \\
If $y \in E$ is such that $s\rho <\mathrm{dist} (y, \partial E) \le \frac{\vartheta^*}{4}$, then, by the same reasoning as in Step 4, we obtain 
\begin{equation}\label{4591}
\frac{\nor{\nabla u}{B_{\rho}(y)}}{\nor{\nabla u}{E}}
\le C \Bigg( \frac{\nor{\nabla u}{B_\rho (\bar{x})}}{\nor{\nabla u}{E}} \Bigg)^{ \delta_\chi^{k(\rho)-1} \delta^{k(\rho)+L-1}}, \end{equation}
with $C>1$ again depending only on $n$ and $M_0$. 
If, on the other hand, $y \in E$ is such that $\mathrm{dist}(y, \partial E) \ge \frac{\vartheta^*}{4}$, taking $\tilde{r}$ as in (\ref{446}), using the same argument as in Step 4, we obtain 
\begin{equation}\label{4601}
\frac{\nor{\nabla u}{B_{\rho}(y)}}{\nor{\nabla u}{E}}
\le C \Bigg( \frac{\nor{\nabla u}{B_\rho(\bar{x})}}{\nor{\nabla u}{E}} \Bigg)
^{\delta_\chi^{k(\rho)-1} \delta^{\tilde{L}}},
\end{equation}
where again $C>1$ only depends on $n$ and $M_0$. 
From (\ref{4591}),(\ref{4601}), (\ref{sferealmassimo}),(\ref{sferealmassimotilde}) and (\ref{432}), and recalling that, again, $\delta_\chi \le \delta$, and $\frac{\nor{\nabla u}{B_{\rho}(y)}}{\nor{\nabla u}{E}} \le 1$, we obtain 
\begin{equation}
\frac{\nor{\nabla u}{B_{\rho}(y)}}{\nor{\nabla u}{E}}
\le C \Bigg( \frac{\nor{\nabla u}{B_\rho (\bar{x})}}{\nor{\nabla u}{E}} \Bigg)^{ \delta_\chi^{A+B\log\frac{1}{\rho}} },
\end{equation}
where $C>1$ and $B>0$ only depend on $n$ and $M_0$, while $A>0$ only depends on $n$, $M_0$, $M_1$. 
The thesis follows from the same cube covering argument as in Step 4. 
\end{proof}
{\bf Conclusion.}  So far, we have proven (\ref{POS}) true for every $\rho \le \bar{\rho}$, and for every $\bar{x} \in E_{s \rho}$, where $\bar{\rho}$ only depends on $M_0$, $M_1$ and $F$. If $\rho > \bar{\rho}$ and $\bar{x} \in E_{s \rho} \subset  E_{s \bar{\rho}}$, then, using what we have shown so far,
\begin{equation} \label{462}
\nor{\nabla u}{B_\rho (\bar{x})} \ge \nor{\nabla u}{B_{\bar{\rho}}(\bar{x})} \ge \tilde{C} \nor{\nabla u}{E},
\end{equation}
where $\tilde{C}$ again only depends on $n$, $M_0$, $M_1$ and $F$. On the other hand, by the regularity hypotheses on $E$, it is easy to show that 
\begin{equation} \label{463}
\rho \le \frac{\mathrm{diam}(\Omega)}{2s} \le \frac{C^*}{2s}
\end{equation}
thus the thesis 
\begin{displaymath}
\int_{B_\rho (\bar{x})} |\nabla u|^2 \ge \frac{C}{\exp \Big[ A \Big(\frac{1}{\rho}\Big)^B\Big] } \int_E |\nabla u|^2 \end{displaymath}
is trivial, if we set \begin{displaymath} 
C = \tilde{C} \exp\Big[ A \Big( \frac{2s}{C^*}\Big)^B \Big].
\end{displaymath}
\end{proof}

\section{Stability of continuation from Cauchy data.}
Throughout this section, we shall again distinguish two domains $\Omega_i= \Omega \setminus \overline{D_i}$ for $i=1,2$, where $D_i$ are two subset of $\Omega$ satisfying (\ref{apriori2bis}) to (\ref{apriori4}). 
We start by putting up some notation. In the following, we shall call 
\begin{displaymath} U^i_\rho =\{x \in \overline{\Omega_i} \; \text{s.t.} \mathrm{dist}(x,\partial \Omega) \le \rho \}. \end{displaymath} 
The following are well known results of interior regularity for the bilaplacian (see, for example, \cite{Miranda}, \cite{GilTru}):
\begin{lemma}[Interior regularity of solutions] \label{teoschauder} Let $u_i$ be the weak solution to \ref{NSE} in $\Omega_i$. Then for all $0<\alpha<1$ we have that $u_i \in C^{1,\alpha}(\overline{\Omega_i \setminus U^i_{\frac{\rho_0}{8}}})$ and 
\begin{equation} \label{schauder1} \|u_i \|_{C^{1,\alpha}(\overline{ \Omega_i \setminus U^i_{\frac{\rho_0}{8}}})} \le C 
\norma{g}{\frac{1}{2}}{\Gamma}  \end{equation} 
\begin{equation} \label{schauder2} \|u_1-u_2 \|_{C^{1,\alpha}( \overline{\Omega_1 \cap \Omega_2})}  \le C 
\norma{g}{\frac{1}{2}}{\Gamma}  \end{equation} where $C>0$ only depends on $\alpha$, $M_0$.  \end{lemma}
\begin{proof}
Using standard energy estimates, as in Theorem \ref{TeoRegGen}, it follows that
\begin{equation} \label{stimau} \norma{u_i}{1}{\Omega_i} \le C 
\norma{g}{\frac{1}{2}}{\partial \Omega}. \end{equation}
On the other hand, using interior regularity estimates for biharmonic functions, we have 
\begin{equation} \label{intreg}
\|u_i \|_{C^{1,\alpha} (\overline{\Omega_i \setminus U^i_{\frac{\rho_0}{8}}})}  \le C \|u_i \|_{\mathbf{L}^{\infty} (\overline{\Omega_i \setminus U^i_{\frac{\rho_0}{16}}})}  \le
\normadue{u_i}{\Omega_i},
\end{equation}
where $C>0$ only depends on $\alpha$ and $M_0$. Combining (\ref{stimau}), (\ref{intreg}), and recalling (\ref{equivalence}), immediately leads to (\ref{schauder1}).
As for (\ref{schauder2}),  we observe that $u_1-u_2=0$ on $\Gamma$ (actually, on $\partial \Omega$); therefore, the $C^{1,\alpha}$ norm of $u_1-u_2$ in $U_{\frac{\rho_0}{2}}^1 \cap U_{\frac{\rho_0}{2}}^2$ can be estimated in the same fashion; using (\ref{schauder1}) in the remaining part, we get (\ref{schauder2}).
\end{proof}
We will also need the following lemma, proved in \cite{ABRV}:
\begin{lemma}[Regularized domains] \label{regularized} 
Let $\Omega$ be a domain satisfying (\ref{apriori1}) and (\ref{apriori2}), and let $D_i$, for $i=1,2$ be two connected open subsets of $\Omega$ satisfying (\ref{apriori3}), (\ref{apriori4}). Then there exist a family of regularized domains $D_i^h \subset \Omega$, for  $0 < h < a \rho_0$, with $C^1$ boundary of constants $\til{\rho_0}$, $\til{M_0}$  and such that 
\begin{equation} \label{643} D_i \subset D_i^{h_1} \subset D_i^{h_2} \; \text{ if  } 0<h_1 \le h_2; \end{equation}
\begin{equation} \label{644} \gamma_0 h \le \mathrm{dist}(x, \partial D_i) \le \gamma_1 h  \; \text{ for all   } x \in  \partial D_i^h; \end{equation}
\begin{equation} \label{645} \mathrm{meas}(D_i^h\setminus D_i)\le \gamma_2 M_1 \rho_0^2 h; \end{equation}
\begin{equation} \label{646} \mathrm{meas}_{n-1}(\partial D_i^h)\le \gamma_3 M_1 \rho_0^2; \end{equation}
and for every $x \in \partial D_i^h$ there exists $y \in \partial D_i$ such that 
\begin{equation} \label{647} |y-x|= \mathrm{dist}(x, \partial D_i), \; \; |\nu(x) - \nu(y)|\le \gamma_4 \frac{h^\alpha}{\rho_0^\alpha}; \end{equation}
where by $\nu(x)$ we mean the outer unit normal to $\partial D_i^h$, $\nu(y)$ is the outer unit normal to $D_i$, and the constants $a$, $\gamma_j$, $j=0 \dots 4$ and the ratios 
$\frac{\til{M}_0}{M_0}$, $\frac{\til{\rho}_0}{\rho_0}$ only depend on $M_0$ and $\alpha$.
\end{lemma}
We shall also need a stability estimate for the Cauchy problem associated with the Stokes system with homogeneous Cauchy data. The proof of the following result, which will be given in the next section, basically revolves around an extension argument. Let us consider a bounded domain $E\subset \mathbb{R}^n$ satisfying hypotheses (\ref{apriori1}) and (\ref{apriori2}), and take $\Gamma \subset \partial E$ a connected open portion of the boundary of  class $C^{2, \alpha}$ with constants $\rho_0$, $M_0$. Let $P_0 \in \Gamma$ such that (\ref{apriori2G}) holds. By definition, after a suitable change of coordinates we have that $P_0 = 0$ and 
\begin{equation} 
E \cap B_{\rho_0}(0) = \{ (x^\prime, x_n) \in E  \, \text{ s.t.} \, x_n>\varphi(x^\prime)  \} \subset E,
\end{equation}
where $\varphi$ is a $C^{2,\alpha}(B^\prime_{\rho_0}(0))$ function satisfying 
\begin{displaymath}
 \begin{split}
  \varphi(0)&=0, \\ 
|\nabla \varphi (0)|&=0, \\ 
\|\varphi \|_{C^{2,\alpha} (B^\prime_{\rho_0}(0))}& \le M_0 \rho_0.
 \end{split}
\end{displaymath}
 Define
\begin{equation} \begin{split} \label{rho00}
 \rho_{00} & = \frac{\rho_0}{\sqrt{1+M_0^2}}, \\ \Gamma_0 & = \{ (x^\prime, x_n)   \in \Gamma \, \, \mathrm{s.t.} \, \, |x^\prime|\le \rho_{00}, \, \,  x_n = \varphi(x^\prime) \}.
\end{split} \end{equation}
\begin{theorem} \label{stabilitycauchy}
Under the above hypotheses, let $(u,p)$ be a solution to the problem:
\begin{equation}
  \label{NseHomDir} \left\{ \begin{array}{rl}
    \dive \sigma(u,p) & = 0 \hspace{2em} \mathrm{\tmop{in}} \hspace{1em}
    E,\\
    \dive u & = 0 \hspace{2em} \mathrm{\tmop{in}} \hspace{1em} E,\\
    u & = 0 \hspace{2em} \mathrm{\tmop{on}} \hspace{1em} \Gamma,\\
    \sigma (u, p) \cdot \nu & = \psi \hspace{2em} \mathrm{\tmop{on}}
    \hspace{1em} \Gamma,\\  \end{array} \right.
\end{equation}
where $\psi \in \accan{-\frac{1}{2}}{\Gamma}$. Let $P^* = P_0 + \frac{\rho_{00}}{4} \nu$ where $\nu$ is the outer normal field to $\partial \Omega$. Then we have 
\begin{equation} \label{NseHomDirEqn}
\| u \|_{{\bf L}^\infty(E \cap B_{\frac{3 \rho_{00}}{8}} (P^*))} \leq \frac{C}{\rho_0^{\frac{n}{2}}} \normadue{u}{E}^{1-\tau} (\rho_0 \norma{\psi}{-\frac{1}{2}}{\Gamma})^\tau,
\end{equation}
where $C>0$ and $\tau$ only depend on $\alpha$ and $M_0$. 
\end{theorem}
\begin{proof}[Proof of Proposition \ref{teostabest}]
Let $\theta= \mathrm{min} \{a, \frac{7}{8 \gamma_1} \frac{\rho_{0}}{2\gamma_0 (1+M_0^2)} \}$ where $a$, $\gamma_0$, $\gamma_1$ are the constants depending only on $M_0$ and $\alpha$ introduced in Lemma \ref{regularized}, then let $\overline{\rho}= \theta \rho_0$ and fix $\rho \le \overline{\rho}$.
We introduce the regularized domains $D_1^\rho$, $D_2^\rho$ according to Lemma \ref{regularized}. Let $G$ be the connected component of $\Omega\setminus(\overline{D_1 \cup D_2})$ which contains $\partial \Omega$, and $G^\rho$ be the connected component of $\overline{\Omega}\setminus(D_1^\rho \cup D_2^\rho)$ which contains $\partial \Omega$.
We have  that \begin{equation*}
D_2 \setminus \overline{D_1} \subset \Omega_1 \setminus \overline{G} \subset \big( (D_1^\rho \setminus \overline{D_1} ) \setminus\overline{G}\big) \cup \big( (\Omega \setminus G^\rho)\setminus D_1^\rho \big)
\end{equation*} 
and 
\begin{equation*}
\partial \big( (\Omega \setminus G^\rho)\setminus D_1^\rho \big) = \Gamma_1^\rho \cup \Gamma_2^\rho,
\end{equation*}
where $\Gamma_2^\rho= \partial D_2^\rho \cap \partial G^\rho$ and $\Gamma_1^\rho \subset \partial D_1^\rho$. It is thus clear that  
\begin{equation} \label{652} \int_{D_2 \setminus \overline{D_1 }} |\nabla u_1|^2 \le \int_{\Omega_1 \setminus \overline{G}} |\nabla u_1|^2 \le \int_{(D_1^\rho \setminus \overline{D_1} )\setminus\overline{G}} |\nabla u_1|^2 +\int_{(\Omega \setminus G^\rho)\setminus D_1^\rho} |\nabla u_1|^2. \end{equation} 
The first summand is easily estimated, for using (\ref{schauder1}) and (\ref{645}) we have 
\begin{equation} \label{6.53} \int_{(D_1^\rho \setminus \overline{D_1} )\setminus\overline{G}} |\nabla u_1|^2 \le C \rho_0^{n-2} \norma{g}{\frac{1}{2}}{\Gamma}^2 \frac{\rho}{\rho_0} \end{equation}
where $C$ only depends on the $M_0$, $M_1$ and $\alpha$.
We call $\Omega(\rho)= (\Omega \setminus G^\rho)\setminus D_1^\rho$. The second term in (\ref{652}), using the divergence theorem twice, becomes:
\begin{equation} \label{sommandi} \begin{split} & \int_{\Omega(\rho)} |\nabla u_1|^2 = \int_{\partial\Omega(\rho)} (\nabla u_1 \cdot \nu) u_1 - \int_{\Omega(\rho)} \triangle u_1 \cdot u_1 = \\& \int_{\partial\Omega(\rho)} (\nabla u_1 \cdot \nu) u_1  - \int_{\Omega(\rho)} \nabla p_1 \cdot u_1 =  \int_{\partial\Omega(\rho)} (\nabla u_1 \cdot \nu) u_1  + \int_{\partial \Omega(\rho)} p_1  (u_1\cdot \nu) = \\ & \int_{\Gamma_1^\rho}(\nabla u_1 \cdot \nu) u_1  + \int_{\Gamma_2^\rho}(\nabla u_1 \cdot \nu) u_1 + \int_{\Gamma_1^\rho}  p_1  (u_1 \cdot \nu) +\int_{\Gamma_2^\rho}  p_1  (u_1 \cdot \nu) . \end{split} \end{equation} 
About the first and third term, if $x \in \Gamma_1^\rho$, using Lemma \ref{regularized}, we find $y \in \partial D_1$ such that $|y-x|= d(x, \partial D_1) \le \gamma_1 \rho$; since $u_1(y)=0$, by Lemma \ref{teoschauder} we have 
\begin{equation} \label{pezzobuono}  |u_1(x)|= |u_1(x)-u_1(y)|\le  C \frac{\rho}{\rho_0}  \norma{g}{\frac{1}{2}}{\Gamma} . \end{equation}
On the other hand, if $x \in \Gamma_2^\rho$, there exists $y \in D_2$ such that $|y-x| = d(x, \partial D_2) \le \gamma_1 \rho$. Again, since $u_2(y)=0$, we have 
\begin{equation} \label{pezzocattivo} \begin{split} & |u_1(x)| \le  |u_1(x)-u_1(y)|+|u_1(y)-u_2(y) |   \\ 
& \le C  \big( \frac{\rho}{\rho_0} \norma{g}{\frac{1}{2}}{\Gamma} + \max_{\partial G^\rho \setminus \partial \Omega} |w| \big) , \end{split}\end{equation} 
where $w=u_1-u_2$.  Combining (\ref{pezzobuono}), (\ref{pezzocattivo}) and (\ref{sommandi}) and recalling (\ref{schauder1}) and (\ref{646}) we have:
\begin{equation} \label{sommandi2}
\int_{D_2\setminus D_1} |\nabla u_1|^2  \le C\rho_0^{n-2} \Big(  \norma{g}{\frac{1}{2}}{\Gamma}^2 \frac{\rho}{\rho_0} + \norma{g}{\frac{1}{2}}{\Gamma}  \max_{\partial G^\rho \setminus \partial \Omega} |w|  \Big)
\end{equation}
We now need to estimate $\max_{\partial G^\rho \setminus \partial \Omega} |w| $.  We may apply (\ref{tresfere}) to $w$, since it is biharmonic. Let $ x \in \partial G^\rho \setminus \partial \Omega$ and \begin{equation} \label{rhostar} \rho^*=\frac{\rho_0}{16(1+M_0^2)}, \end{equation}
\begin{equation}\label{zetazero}
x_0= P_0 - \frac{\rho_1}{16}\nu,
\end{equation}
where $\nu$ is the outer normal to $\partial \Omega$ at the point $P_0$. By construction $x_0 \in \overline{\til{\Omega}_{\frac{\rho^*}{2}}}$. 
There exists an arc $\gamma:  [0,1]  \mapsto G^\rho \setminus \overline{\til{\Omega}_{\frac{\rho^*}{2}}} $ such that $\gamma(0)=x_0$, $\gamma(1)=x$ and $\gamma([0,1])\subset G^\rho \setminus \overline{\til{\Omega}_{\frac{\rho^*}{2}}}$. Let us define a sequence of points $\{x_i \}_{i=0 \dots S}$ as follows: $t_0=0$, and
\begin{displaymath} 
t_i= \mathrm{max} \{t\in (0,1]  \text{  such that  } |\gamma(t)- x_i| = \frac{\gamma_0 \rho \vartheta^*}{2} \} \; \text{, if } |x_i-x| >\frac{\gamma_0 \rho \vartheta^*}{2}, 
\end{displaymath}
otherwise, let $i=S$ and the process is stopped. Here $\vartheta^*$ is the constant given in Theorem \ref{teotresfere}. All the balls $B_{\frac{\gamma_0 \rho \vartheta^*}{4}}(x_i)$ are pairwise disjoint, the distance between centers $| x_{i+1}-x_i | =\frac{\gamma_0 \rho \vartheta^*}{2}$ for all $i=1 \dots S-1$ and for the last point, $|x_S - x| \le \frac{\gamma_0 \rho \vartheta^*}{2}$. The number of spheres is bounded by  
\begin{displaymath} S\le C \Big( \frac{\rho_0}{\rho} \Big)^n \end{displaymath} where $C$ only depends on $\alpha$, $M_0$ and $M_1$. For every $\rho \le \overline{\rho}$, we have that, letting 
\begin{displaymath} \rho_1 = \frac{\gamma_0 \rho \vartheta^*}{4},\;  \rho_2= \frac{3 \gamma_0 \rho \vartheta^*}{4}, \;  \rho_3={\gamma_0 \rho \vartheta^*} 
\end{displaymath}
an iteration of the three spheres inequality on a chain of spheres leads to 
\begin{equation} \label{iteratresfere}  \int_{B_{\rho_2} (x)} \! | w|^2 dx \le C \Big(\int_G  \! | w|^2 dx \Big)^{1-\delta^S}  \Big(\int_{B_{\rho_3}(x_0)} \! | w |^2 dx \Big)^{\delta^S} \end{equation}
where $0<\delta<1$ and $C>0$ only depend on $M_0$ and $\alpha$. From our choice of $\bar{\rho}$ and $\vartheta^*$, it follows that $B_{\frac{\gamma_0 \rho \vartheta^*}{4}}(x_0) \subset B_{\rho^*}(x_0) \subset G \cap B_{\frac{3 \rho_1 }{4}}(P^*)$, where we follow the notations from Theorem \ref{stabilitycauchy}. We can therefore apply Theorem \ref{stabilitycauchy}. Let us call 
\begin{equation}  \label{epsilontilde}
\tilde{\epsilon} = \frac{ \epsilon}{ \norma{g}{\frac{1}{2}}{\Gamma} }.
\end{equation}
Using (\ref{NseHomDirEqn}), (\ref{stimau}) and (\ref{HpPiccolo}) on (\ref{iteratresfere}) we then have:
\begin{equation} \label{pallina}
 \int_{B_{\rho_2}(x)} \! | w|^2 dx \le C \rho_0^{n-2} \norma{g}{\frac{1}{2}}{\Gamma}^2 \tilde{\epsilon}^{2 \tau \delta^S}.
\end{equation}
The following interpolation inequality holds for all functions $v$ defined on the ball $B_t(x) \subset \mathbb{R}^n$:
\begin{equation} \label{interpolation}
\|v \|_{\mathbf{L}^\infty (B_t(x))} \le C \Big( \Big(  \int_{B_t(x)} | v|^2 \Big)^{\frac{1}{n+2}} |\nabla v|^{\frac{n}{n+2}}_{\mathbf{L}^\infty (B_t(x))} +   \frac{1}{t^{n/2}} \Big( \int_{B_t(x)} | v|^2 \Big)^{\frac{1}{2}} \Big)
\end{equation}
We apply it to $w$ in $B_{\rho_2}(x)$, using (\ref{pallina}) and (\ref{schauder1}) we obtain
\begin{equation} \label{stimaw}
\| w \|_{\mathbf{L}^\infty (B_{\rho_2}(x))} \le C \Big( \frac{\rho_0}{\rho} \Big)^{\frac{n}{2}}  \norma{g}{\frac{1}{2}}{\Omega} \tilde{\epsilon}^{\gamma  \delta^S},
\end{equation}
where $\gamma=\frac{2\tau}{n+2}$. 
Finally, from (\ref{stimaw}) and (\ref{sommandi2}) we get:
\begin{equation} \label{sommandi3} 
\int_{D_2\setminus D_1} |\nabla u_1|^2  \le C \rho_0^{n-2} \norma{g}{\frac{1}{2}}{\Gamma}^2 \Big( \frac{\rho}{\rho_0} + \Big( \frac{\rho_0}{\rho} \Big)^{\frac{n}{2}} \tilde{\epsilon}^{\gamma \delta^S}  \Big)
\end{equation}
Now call $$\til{\mu}=\exp \Big( -\frac{1}{\gamma} \exp \Big(\frac{2S \log \delta}{\theta^n}\Big)\Big) $$ and $\overline{\mu}= \min \{ \til{\mu}, \exp(-\gamma^2) \}.$
Choose $\rho$ depending upon $\tilde{\epsilon}$ of the form 
\begin{displaymath}
 \rho(\tilde{\epsilon}) = \rho_0 \Bigg( \frac{2S \log |\delta|}{\log |\log \tilde{\epsilon}^\gamma|} \Bigg)^{-\frac{1}{n}}.
\end{displaymath}
We have that $\rho$ is defined and increasing in the interval $(0, e^{-1})$, and by definition $\rho(\overline{\mu}) \le \rho(\til{\mu}) = \theta \rho= \overline{\rho}$, we are able to apply (\ref{sommandi3}) to (\ref{652}) with $\rho=\rho(\til{\epsilon})$ to obtain 
\begin{equation}
\label{quasifinito}
\int_{D_2 \setminus D_1} |\nabla u_1|^2 \le C \rho_0^{n-2} \norma{g}{\frac{1}{2}}{\Gamma}^2 \log |\log \til{\epsilon}|^\gamma,
\end{equation}
and since $\til{\epsilon} \le \exp(-\gamma^2)$  it is elementary to prove that \begin{displaymath}
\log |\log {\til{\epsilon}^\gamma}| \ge \frac{1}{2} \log | \log \til{\epsilon}|,
\end{displaymath}
so that (\ref{quasifinito}) finally reads 
\begin{displaymath}
\int_{D_2 \setminus D_1} |\nabla u_1|^2 \le C \rho_0^{n-2} \norma{g}{\frac{1}{2}}{\Gamma}^2 \,\omega(\til{\epsilon}),
\end{displaymath}
with  $\omega(t) = \log |\log t|^{\frac{1}{n}}$ defined for all $0<t<e^{-1}$, and $C$ depends on $M_0$, $M_1$ and $\alpha$.
\end{proof}
\begin{proof}[Proof of Proposition \ref{teostabestimpr}]
We will prove the thesis for $u_1$, the case $u_2$ being completely analogous. 
First of all, we observe that
\begin{equation} \label{sommandiB}
\int_{D_2 \setminus D_1} |\nabla u_1|^2 \le \int_{\Omega_1 \setminus G} |\nabla u_1|^2 =\int_{\partial (\Omega_1 \setminus G)} (\nabla u_1 \cdot \nu) u_1 + \int_{\partial (\Omega_1 \setminus G)} p_1 ( u_1 \cdot \nu)  
\end{equation}
and that 
\begin{equation*}
\partial (\Omega_1 \setminus G) \subset \partial D_1 \cup (\partial D_2 \cap \partial G)
\end{equation*}
and recalling the no-slip condition, applying to (\ref{sommandiB}) computations similar to those in (\ref{652}), (\ref{6.53}), we have
\begin{equation*} \begin{split}
& \int_{D_2 \setminus D_1} |\nabla u_1|^2  \le \int_{\partial D_2 \cap \partial G} (\nabla u_1 \cdot \nu) w + \int_{\partial D_2 \cap \partial G} p_1 ( w \cdot \nu) \le \\ \le & C \rho_0^{n-2}\norma{g}{\frac{1}{2}}{\Gamma} \max_{\partial D_2 \cap \partial G} |w|,
\end{split} \end{equation*}
where again $w= u_1 - u_2$ and $C$ only depends on $\alpha$, $M_0$ and $M_1$.
Take a point $z \in \partial G$. By the regularity assumptions on $\partial G$,   we find a direction $\xi \in \mathbb{R}^n$, with $|\xi|=1$, such that the cone (recalling the notations used during the proof of Proposition \ref{teoPOS}) $C(z, \xi, \vartheta_0) \cap B_{\rho_0} (z) \subset G$, where $\vartheta_0 =\arctan \frac{1}{M_0}$. Again (\cite[Proposition 5.5]{ARRV}) $G_\rho$ is connected for $\rho \le \frac{\rho_0 h_0 }{3}$ with $h_0$ only depending on $M_0$. Now set
\begin{equation*}\begin{split}
\lambda_1 &= \min \Big\{ \frac{\tilde{\rho}_0}{1+\sin \vartheta_0}, \frac{\tilde{\rho}_0}{3\sin \vartheta_0}, \frac{\tilde{\rho}_0}{16(1+M_0^2)\sin \vartheta_0} \frac{}{} \Big\}, \\ 
\vartheta_1 & = \arcsin\Big(\frac{\sin \vartheta_0}{4} \Big), \\
w_1 &=z+ \lambda_1 \xi, \\ 
\rho_1 &= \vartheta^* h_0 \lambda_1 \sin \vartheta_1. 
\end{split}\end{equation*}
where $0<\vartheta^*\le 1$ was introduced in Theorem \ref{teotresfere}.
By construction, $B_{\rho_1}(w_1) \subset C(z, \xi, \vartheta_1) \cap B_{\tilde{\rho}_0}(z)$ and $B_{\frac{4 \rho_1}{\vartheta^*}}(w_1) \subset C(z, \xi, \vartheta_0) \cap B_{\tilde{\rho}_0}(z) \subset G$. Furthermore $\frac{4 \rho_1}{\vartheta^*} \le  \rho^*$, hence $B_{\frac{4 \rho_1}{\vartheta^*}} \subset G$, where $\rho^*$ and $x_0$ were defined by (\ref{rhostar}) and (\ref{zetazero}) respectively, during the previous proof. Therefore, $w_1$, $x_0 \in \overline{G_{\frac{4\rho_1}{\vartheta^*}}}$, which is connected by construction.
Iterating the three spheres inequality (mimicking the construction made in the previous proof) 
\begin{equation} \label{iteratresferei}  \int_{B_{\rho_1} (w_1)} \! | w|^2 dx \le C \Big(\int_G  \! | w|^2 dx \Big)^{1-\delta^S}  \Big(\int_{B_{\rho_1 }(x_0)} \! | w |^2 dx \Big)^{\delta^S} \end{equation}
where $0<\delta<1$ and $C \ge 1$ depend only on $n$, and $S \le \frac{M_1 \rho_0^n}{\omega_n \rho_1^n}$. 
Again, since $B_{\rho^*}(x_0) \subset G \cap B_{\frac{3}{8}\rho_1}(P_0)$, we apply Theorem \ref{stabilitycauchy} which leads to 
\begin{equation}
\int_{B_{\rho_1}(w_1)} |w|^2 \le C \rho_0^n \norma{g}{\frac{1}{2}}{\Gamma}^2 \tilde{\epsilon}^{2\beta},
\end{equation}
where $0<\beta<1$ and $C \ge 1$ only depend on $\alpha$, $M_0$, and $\frac{\tilde{\rho}_0}{\rho_0}$ and $\tilde{\epsilon}$ was defined in (\ref{epsilontilde}).
So far the estimate we have is only on a ball centered in $w_1$, we need to approach $z \in \partial G$ using a sequence of balls, all contained in $C(z, \xi, \vartheta_1)$, by suitably shrinking their radii. Take 
\begin{equation*}
\chi = \frac{1-\sin \vartheta_1}{1+\sin\vartheta_1}
\end{equation*}
and define, for $k \ge 2$,
\begin{equation*} \begin{split}
\lambda_k&=\chi \lambda_{k-1}, \\
\rho_k&= \chi \rho_{k-1}, \\
w_k &= z + \lambda_k \xi. \\
\end{split}\end{equation*}
With these choices, $\lambda_k= \lambda \chi^{k-1} \lambda_1$, $\rho_k=\chi^{k-1} \rho_1$ and $B_{\rho_{k+1}}(w_{k+1}) \subset B_{3\rho_k}(w_k)$, $B_{\frac{4}{\vartheta^*}\rho_k}(w_k) \subset C(z, \xi, \vartheta_0) \cap B_{\tilde{\rho}_0}(z) \subset G$.
Denote by
\begin{displaymath}
d(k)= |w_k-z|-\rho_k,
\end{displaymath}
we also have
\begin{displaymath}
d(k)= \chi^{k-1}d(1),
\end{displaymath}
with
\begin{displaymath}
d(1)= \lambda_1(1-\vartheta^* \sin \vartheta_1).
\end{displaymath}
Now take any $\rho \le d(1)$ and let $k=k(\rho)$ the smallest integer such that $d(k) \le \rho$, explicitly
\begin{equation} \label{chirho}
\frac{\big|\log \frac{\rho}{d(1)}\big|}{\log \chi} \le k(\rho)-1 \le \frac{| \log \frac{\rho} {d(1)}|}{\log \chi}+1.
\end{equation}
We iterate the three spheres inequality over the chain of balls centered in $w_j$ and radii $\rho_j$, $3 \rho_j$, $4\rho_j$, for $j=1, \dots, k(\rho)-1$, which yields
\begin{equation} \label{iteratresferetre}
\int_{B_{\rho_{k(\rho)}}(w_{k(\rho)})} |w|^2 \le C \norma{g}{\frac{1}{2}}{\Gamma}^2 \rho^n \tilde{\epsilon}^{2 \beta \delta^{k(\rho)-1}},
\end{equation}
with $C$ only depending on $\alpha$, $M_0$ and $\frac{\tilde{\rho}_0}{\rho_0}$.
Using the interpolation inequality (\ref{interpolation}) and (\ref{schauder2}) we obtain 
\begin{equation}\label{543}
\|w \|_{\mathbf{L}^\infty (B_{\rho_{k(\rho)}}(w_{k(\rho)}))} \le C \norma{g}{\frac{1}{2}}{\Gamma} \frac{\tilde{\epsilon}^{\beta_1 \delta^{k(\rho)-1}}}{\chi^{\frac{n}{2}(k(\rho)-1)}},
\end{equation}
where $\beta_1=\frac{2 \beta}{n+2}$ depends only on $\alpha$, $M_0$, $M_1$ and $\frac{\tilde{\rho}_0}{\rho_0}$.
From (\ref{543}) and (\ref{schauder2}) we obtain 
\begin{equation} \label{544}
|w(z) | \le C \norma{g}{\frac{1}{2}}{\Gamma} \Bigg( \frac{\rho}{\rho_0} +\frac{\tilde{\epsilon}^{\beta_1 \delta^{k(\rho)-1}}}{\chi^{\frac{n}{2}(k(\rho)-1)}} \Bigg),
\end{equation}
Finally, call
\begin{displaymath}
\rho(\tilde{\epsilon})= d(1) |\log \tilde{\epsilon}^{\beta_1}|^{-B},
\end{displaymath}
with
\begin{displaymath}
B= \frac{|\log \chi|}{2 \log |\delta|}.
\end{displaymath}
and let $\tilde{\mu} = \exp(-\beta_1^{-1})$. We have that $\rho(\tilde{\epsilon})$ is monotone increasing in the interval $0<\tilde{\epsilon} < \tilde{\mu}$, and $\rho(\tilde{\mu})=d(1)$, so $\rho(\tilde{\epsilon}) \le d(1)$ there.  Putting $\rho=\rho(\tilde{\epsilon})$ into (\ref{544}) we obtain
\begin{equation}
\int_{D_2 \setminus D_1} |\nabla u_1|^2 \le C \rho_0^{n-2}\norma{g}{\frac{1}{2}}{\Gamma}^2 |\log \tilde{\epsilon}|^{-B},
\end{equation}
where $C$ only depends on $\alpha$, $M_0$ and $\frac{\til{\rho}_0}{\rho_0}$.
\end{proof}

\section{Proof of Theorem \ref{stabilitycauchy}. }
As already premised, in order to prove Theorem \ref{stabilitycauchy}, we will need to perform an extension argument on the solution to (\ref{NSE}) we wish to estimate. This has been done for solutions to scalar elliptic equations with sufficiently smooth coefficients (\cite{Isa}). Here, however, we are dealing with a system: extending $u$ implies finding a suitable extension for the pressure $p$ as well; moreover, both extensions should preserve some regularity they inherit from the original functions. 
Following the notations given for Theorem \ref{stabilitycauchy} we define  $$Q(P_0) = B^\prime_{\rho_{00}} (0) \times \Big[-\frac{M_0\rho_0^2}{\sqrt{1+M_0^2}}, \frac{M_0\rho_0^2}{\sqrt{1+M_0^2}}\Big].$$  
We have: 
\begin{equation}\label{gagrafico} \begin{split}
\Gamma_0 &= \partial E \cap Q(P_0). \\ 
\end{split}\end{equation}
We then call $E^- = Q(P_0) \setminus  E$ and $\til{E} = E \cup E^- \cup \Gamma_0$.
\begin{lemma}[Extension] \label{teoextensionNSE}
Suppose the hypotheses of Theorem \ref{stabilitycauchy} hold.  Consider the domains $E^-$, $\til{E}$ as constructed above.  Take, furthermore, $g \in \accan{\frac{5}{2}}{\partial E}$. Let $(u,p)$ be the solution to the following problem:
\begin{equation}
  \label{NseHomDirExt} \left\{ \begin{array}{rl}
    \dive \sigma(u,p) & = 0 \hspace{2em} \mathrm{\tmop{in}} \hspace{1em}
    E,\\
    \dive u & = 0 \hspace{2em} \mathrm{\tmop{in}} \hspace{1em} E,\\
    u & = g \hspace{2em} \mathrm{\tmop{on}} \hspace{1em} \Gamma,\\
    \sigma (u, p) \cdot \nu & = \psi \hspace{2em} \mathrm{\tmop{on}}
    \hspace{1em} \Gamma,\\  \end{array} \right.
\end{equation}
Then there exist functions $\tilde{u} \in \accan{1}{\til{E}}$,  $\til{p} \in L^2(\til{E})$ and a functional $\Phi \in \accan{-1}{\til{E}}$ such that $\tilde{u} = u$,  $\tilde{p} = p$ in $E$ and $(\til{u}, \til{p})$ solve the following:
\begin{equation} \begin{split} \label{sistilde}
\triangle \til{u} + \nabla \til{p} &= \Phi \, \, \text{  in  } \, \, \til{E},  \\ \dive \til{u}&=0  \, \, \text{  in  } \, \, \til{E}.
\end{split}
\end{equation}
If
\[ \norma{g}{\frac{1}{2}}{\Gamma}+ \rho_0\norma{\psi}{-\frac{1}{2} }{\Gamma} = \eta , \] 
then we have 
\begin{equation} \label{stimaPhi}
\norma{\Phi}{-1}{\til{E}} \le C\frac{\eta}{\rho_0}.
\end{equation}
where $C>0$ only depends on $\alpha$ and $M_0$. 
\end{lemma}
\begin{proof}
From the assumptions we made on the boundary data and the domain, it follows that $(u, p) \in \accan{3}{E} \times L^2(E) $.
We can find (see \cite{MITREA} or \cite{BedFix}) a function $u^- \in \accan{3}{E^-}$  such that
\begin{equation} \label{propumeno} \begin{split} \dive u^-=0 \quad \mathrm{in}\quad E^-,\qquad u^-=g \quad\mathrm{on} \quad \Gamma,\\  \norma{u^-}{3}{E^-} \le C \norma{g}{\frac{1}{2}}{\Gamma},
 \end{split} \end{equation}
with $C$ only depending on $|E|$. 
We now call 
\begin{displaymath} F^-= \triangle u^-,\end{displaymath} 
by our assumptions we have $ F^- \in \accan{1}{E^-}$.
Let $p^- \in H^1(E^-)$ be the weak solution to the following Dirichlet problem:\begin{equation}\label{pmeno} \left\{ \begin{array}{rl}
    \triangle p^- - \dive F^- &=0 \hspace{2em} \mathrm{\tmop{in}} \hspace{1em} E^-,\\
    p^- & = 0  \hspace{1em}
    \hspace{0.50em}  \mathrm{\tmop{on}} \hspace{1em}
    \partial E^-.\\
    \end{array} \right.
\end{equation}
We now define \begin{equation}\label{effestesa} X^-= F^- -\nabla p^-. \end{equation} This field is divergence free by construction,
and its norm is controlled by \begin{equation} \label{stimaX} \|X^- \|_{\elledue{E^-}} \le C  \norma{g}{\frac{1}{2}}{\Gamma} \end{equation}
We thus extend $(u,p)$ as follows:
\begin{displaymath}  \til{u}= \left\{ \begin{array}{rl} & u \quad \text{ in } \; \; E, \\ & u^- \quad \text{ in } \; \; E^-,\end{array} \right.\end{displaymath} \begin{displaymath}\til{p}= \left\{ \begin{array}{rl} &  p \quad \text{ in } \; E, \\ & p^- \quad \text{ in } \; E^-. \end{array} \right. \end{displaymath}
We now investigate the properties of the thus built extension $(\til{u},\til{p})$. Take any $v \in \accano{1}{\til{E}}$, we have
\begin{equation}
  \label{NSEEXT} \begin{split} &\int_{\til{E}} (\nabla \til{u} +(\nabla \til{u})^T - \til{p} \ide  ) \cdot \nabla v = \\ =& \int_{E} (\nabla u +(\nabla u )^T - p \ide ) \cdot \nabla v + \int_{E^-} (\nabla u^- +(\nabla u^-)^T - p^- \ide  ) \cdot \nabla v. \end{split}\end{equation}
About the first term, using (\ref{NSE}) and the divergence theorem we obtain 
\begin{equation}
  \label{phiuno}  \int_{E} (\nabla u +(\nabla u)^T - p \ide ) \cdot \nabla v = \int_{\Gamma} \psi \cdot v. \end{equation}
Define $ \Phi_1(v)= \int_{\Gamma} \psi \cdot v $ for all  $v \in \accano{1}{\til{E}}$.
Using the decomposition made in (\ref{effestesa}) on the second term, we have 
\begin{equation}\label{phiduetre} \begin{split}  & \int_{E^-} (\nabla u^- +(\nabla u^- )^T - p^- \ide  ) \cdot \nabla v = \\ =& \int_{\Gamma}(\nabla u^- +(\nabla u^- )^T - p^- \ide ) \cdot \nu \, v -\int_{E^-} \dive  \big( \nabla u^- +(\nabla u^- )^T - p^-\ide  \big) \cdot  v= \\=& \int_{\Gamma}(\nabla u^- +(\nabla u^- )^T ) \cdot \nu \, v -\int_{E^-} (\triangle u^- -\nabla p^- )  \cdot  v= \\ =&\int_{\Gamma}(\nabla u^- +(\nabla u^- )^T ) \cdot \nu \, v -\int_{E^-} X^-  \cdot  v = \Phi_2(v)+\Phi_3(v), \end{split}\end{equation}
where we define for all $v \in \accano{1}{\til{E}}$ the functionals
\begin{displaymath}
 \begin{split}
\Phi_2(v)&=\int_{\Gamma}(\nabla u^- +(\nabla u^- )^T ) \cdot \nu \, v, \\
\Phi_3(v)&=-\int_{E^-} X^-  \cdot  v 
\end{split}
\end{displaymath}
We can estimate each of the linear functionals $\Phi_1$, $\Phi_2$ and $\Phi_3$ easily, for we have (by (\ref{phiuno}) and the trace theorem):
\begin{equation} \label{stimaphi1} \big| \Phi_1(v) \big| \le \norma{\psi}{-\frac{1}{2}}{\Gamma} \norma{v}{\frac{1}{2}}{\Gamma} \le C  \rho_0\norma{\psi}{-\frac{1}{2}}{\Gamma} \norma{v}{1}{E^-}, \end{equation}
moreover (using (\ref{phiduetre}) and (\ref{propumeno}) )
\begin{equation} \label{stimaphi2} \big| \Phi_2(v) \big| \le {\| \nabla u \|_{{\bf L}^2(\Gamma)}} {\|v \|_{{\bf L}^2(\Gamma)}} \le C \norma{g}{\frac{1}{2}}{\Gamma} \norma{v}{1}{E^-}, \end{equation}
and, at last, by (\ref{stimaX}), 
\begin{equation} \label{stimaphi3} \big| \Phi_3(v) \big| \le \|X^- \|_{\elledue{E^-}} \|v \|_{\elledue{E^-}} \le C \norma{g}{\frac{1}{2}}{\Gamma} \norma{v}{1}{E^-}. \end{equation}
  Then, defining $\Phi(v)=\Phi_1(v) + \Phi_2(v) + \Phi_3(v)$  for all $v \in \accano{1}{\til{E}}$, putting together (\ref{phiuno}), (\ref{phiduetre}), (\ref{stimaphi1}), (\ref{stimaphi2}) and (\ref{stimaphi3}), we have (\ref{stimaPhi}). 
\end{proof}

\begin{proof}[Proof of Theorem \ref{stabilitycauchy}. ]
Consider the domain $\til{E}$ built at the beginning of this section, and take $\til{u}$ the extension of $u$ built according to Theorem \ref{teoextensionNSE}. By linearity, we may write $\til{u}= u_0+w$  where $(w,q)$ solves
\begin{equation}  \label{NSEPARTIC} 
\dive \sigma (w, q)  = \til{\Phi} \hspace{2em} \mathrm{\tmop{in}}
\hspace{1em} \til{E}, \end{equation}
and $w \in \accano{1}{\til{E}}$, whereas $(u_0, p_0)$ solves
\begin{equation}  \label{NSEHOM} 
\left\{ \begin{array}{rl}
    \dive \sigma (u_0, p_0) &= 0 \hspace{2em} \mathrm{\tmop{in}} \hspace{1em} \til{E}, \\ 
    u_0 & = 0 \hspace{2em} \mathrm{\tmop{on}} \hspace{1em} \Gamma,\\
    \sigma (u_0, p_0) \cdot \nu & = \psi \hspace{2em} \mathrm{\tmop{on}}
    \hspace{1em} \Gamma.
  \end{array} \right.
  \end{equation}
Using well known results about interior regularity of solutions to strongly elliptic equations 
\begin{equation} 
 \| u_0 \|_{{\bf L}^\infty( B_{\frac{t}{2}} (x))} \le t^{-\frac{n}{2}} \normadue{u_0}{B_{\frac{t}{2}}(x)}.
\end{equation}
It is then sufficient to estimate $\normadue{u}{B(x)}$ for a "large enough" ball near the boundary. 
Since (see the proof of Proposition \ref{teoPOS})  $\triangle^2 u_0=0$, we  may apply Theorem \ref{teotresfere} to $u_0$. Calling $r_1= \frac{\rho_{00}}{8}$, $r_2= \frac{3 \rho_{00}}{8}$  and $r_3= \rho_{00}$  we have (understanding that all balls are centered in $P^*$) 
\begin{equation} \label{3sfereu0} 
\normadue{u_0}{B_{r_2}} \le C \normadue{u_0}{B_{r_1}}^{\tau} \normadue{u_0}{B_{r_3}}^{1-\tau}.
\end{equation}
Let us call $\eta=\rho_0\norma{\psi}{-\frac{1}{2}}{\Gamma}$.
By the triangle inequality, (\ref{propumeno}) and (\ref{stimau}) we have that
\begin{equation} \label{trin1}
\normadue{u_0}{B_{r}} \le \normadue{\til{u}}{B_{r}}+\normadue{w}{B_{r}} \le \normadue{\til{u}}{B_{r}}  + C \eta, 
\end{equation}
for $r=r_1,r_3$; furthermore, we have 
\begin{equation} \label{trin2}
\normadue{\til{u}}{B_{r_2}} \le \normadue{u_0}{B_{r_2}}+\normadue{w}{B_{r_2}} \le \normadue{u_0}{B_{r_2}}  + C \eta. 
\end{equation}
Putting together (\ref{3sfereu0}), (\ref{trin1}), (\ref{trin2}), and recalling (\ref{stimau}) and (\ref{stimanormadiretto}) we get
\begin{equation} \begin{split} \label{3sfere2}
& \normadue{u}{B_{r_2}} \le \normadue{\til{u}}{B_{r_2} \cap E} \le \\  \le & C \eta + C (\normadue{\til{u}}{B_{r_1}}+ C \eta)^{\tau}  (\normadue{\til{u}}{B_{r_3} \cap E} + C \eta )^{1-\tau} \le \\ \le & C \big( \eta + \eta^\tau (\eta + \normadue{u}{E} )^{1-\tau} \big) \le C \eta^\tau \normadue{u}{E}^{1-\tau}.  \end{split} \end{equation} 
\end{proof}

\end{document}